\documentclass[a4paper,final,reqno]{amsart}%
\usepackage[T2A]{fontenc}
\usepackage[cp1251]{inputenc}
\usepackage[english,russian]{babel}
\usepackage{indentfirst}
\usepackage{misccorr}
\usepackage{amsfonts}
\usepackage{amsmath}
\usepackage{amssymb}
\usepackage{graphicx}%
\providecommand{\U}[1]{\protect\rule{.1in}{.1in}}
\theoremstyle{plain}
\newtheorem{theorem}{Теорема}

\newtheorem{corollary}[theorem]{Следствие}

\newtheorem{lemma}[theorem]{Лемма}

\newtheorem{proposition}[theorem]{Предложение}

\numberwithin{equation}{section}
\numberwithin{theorem}{section}

\begin{document}
\title[Топологические
радикалы]{Топологические
радикалы, III. Cовместный
спектральный
радиус}
\author{Ю. В.
Туровский и
В. С. Шульман}
\address{Институт
математики и
механики
Национальной
Академии
Наук
Азербайджана%
\newline
ул. Ф. Агаева, 9,
Баку AZ1141,
Азербайджан%
\newline
\textit{Эл.адрес:
yuri.turovskii@gmail.com}%
\newline
\indent Вологодский
государственный
технический
университет%
\newline
ул. Ленина, 15,
Вологда 160000,
Россия%
\newline
\textit{Эл.адрес:
shulman.victor80@gmail.com}}
\thanks{2010 \textit{AMS Classif}. Primary 47D03; Secondary 46H05}
\thanks{\textit{Ключевые
слова и
фразы: }%
совместный
спектральный
радиус,
формула
Бергера-Вонга,
топологический
радикал,
инвариантное
подпространство}
\date{}
\dedicatory{Памяти наших
отцов ---
Владимира
Васильевича
Туровского и
Семена
Моисеевича
Шульмана,
офицеров
Советской
армии,
участников
Великой
Отечественной
войны}
\begin{abstract}
Показано,
что
совместный
спектральный
радиус $\rho(M)$
предкомпактного
множества $M$
операторов в
банаховом
пространстве
равен
большему из
двух чисел:
совместного
спектрального
радиуса $\rho_{e}(M)$
образа $M$ в
алгебре
Калкина и
BW-радиуса $r(M)$.
Получены
также
результаты
этого рода,
относящиеся
к общим
нормированным
алгебрам.
Доказательства
основаны на
теории
топологических
радикалов
нормированных
алгебр.

\end{abstract}
\maketitle

\section{Введение}

В 1960 году Рота и
Стрэнг \cite{RS}
ввели
понятие
\textit{совместного
спектрального
радиуса }$\rho$
(сокращенно,
ССР), полагая
для
ограниченного
множества $M$
элементов
нормированной
алгебры $A$,
\begin{equation}
\rho(M):=\underset{n\rightarrow\infty}{\lim}\left\Vert M^{n}\right\Vert
^{1/n}=\inf_{n}\left\Vert M^{n}\right\Vert ^{1/n}, \label{nach}%
\end{equation}
где норма
множества
понимается
как супремум
норм
элементов, а
произведение
множеств ---
как
множество
произведений:
$M_{1}M_{2}=\{ab:a\in M_{1},b\in M_{2}\}$.
Отображение
$M\longmapsto\rho(M)$ обладает
многими
удобными
аналитическими
и
алгебраическими
свойствами;
в частности,
при
аналитической
зависимости
$\lambda\longmapsto M=M\left(  \lambda\right)  $
отображение
$\lambda$ $\longmapsto\rho(M\left(  \lambda\right)  )$
субгармонично
\cite[теорема 3.5]{ShT2000}
(для
конечных
множеств это
было
доказано в
\cite[теорема 3.8]{T85}).

Понятие ССР
нашло
применение в
различных
областях
математики:
эволюционной
динамике,
теории
вэйвлетов,
разностных
уравнениях и
во многих
других (см.
например \cite{Jun}), а
также в
теории
операторов,
в частности,
в теории
инвариантных
подпространств
\cite{Sh84,Tur98,ShT2000}. Это
стимулировало
интерес к
нахождению
формул для
вычисления
ССР. Важная
формула для
вычисления
ССР
ограниченного
множества
матриц была
найдена в 1992
году
Бергером и
Вонгом \cite{BW}.
Чтобы её
привести,
введём,
следуя \cite{BW},
иную
спектральную
характеристику
ограниченного
подмножества
$M$ алгебры ---
\textit{BW-радиус}:
\begin{equation}
r(M):=\underset{n\rightarrow\infty}{\lim\sup\;}r_{n}\left(  M\right)  \text{,
где }r_{n}\left(  M\right)  =\sup\left\{  \rho\left(
a\right)  :a\in M^{n}\right\}  ^{1/n}. \label{nach2}%
\end{equation}
Ясно, что $r(M)\leq\rho(M)$. В
\cite{BW} было
доказано,
что
\begin{equation}
\rho(M)=r(M) \label{bw}%
\end{equation}
для любого
ограниченного
множества матриц.

Формула
Бергера-Вонга
(\ref{bw}) была
перенесена в
\cite{ShT2000} на
предкомпактные
подмножества
компактных
операторов в
банаховом
пространстве
и
использовалась
при изучении
операторных
полугрупп и
алгебр Ли.
Чтобы
убедиться в
её
полезности,
достаточно
увидеть, как
легко из нее
следует
(полученное
в \cite{Tur98}) решение
проблемы
вольтерровой
полугруппы
"\textit{всякая ли
полугруппа
вольтерровых
операторов
имеет
инвариантное
подпространство?}%
". В самом деле,
если $G$ ---
полугруппа
\textit{вольтерровых}
(то есть,
компактных
квазинильпотентных)
операторов,
то $r(M)=0$ для
любого
конечного
множества $M\subset G$.
Из (\ref{bw}) следует,
что $\rho(M)=0$ и тогда
линейная
оболочка
множества $M$
состоит из
вольтерровых
операторов
(см. например
\cite[предложение
3.5]{T85}). Значит,
линейная
оболочка
полугруппы $G$ ---
алгебра
вольтерровых
операторов,
и, по теореме
Ломоносова
\cite{Lom}, она имеет
инвариантное
подпространство.

Подчеркнём,
что оба
ограничения
компактности
в (\ref{bw})
существенны.
В \cite{Gu} приведен
пример пары
операторов $a,b$,
для которых
$r(\{a,b\})=0\neq\rho(\{a,b\}).$
Известны и
примеры
ограниченных
множеств
компактных
операторов,
для которых
равенство (\ref{bw})
не выполнено
(см. например
\cite[предложение
2.16]{PW}).

Мы получим
вариант
формулы
Бергера-Вонга,
свободный от
ограничения
компактности
операторов.
А именно,
будет
показано,
что для
любого
предкомпактного
множества $M\subset
\mathcal{B}\left(  \mathcal{X}\right)  $
справедливо
равенство
\begin{equation}
\rho(M)=\max\{r(M),\rho_{e}(M)\}, \label{gbwf}%
\end{equation}
где $\rho_{e}(M):=\rho(\pi(M))$ --- ССР
образа $M$ в
алгебре
Калкина $\mathcal{B}\left(
\mathcal{X}\right)  /\mathcal{K}\left(  \mathcal{X}\right)  $
при
каноническом
эндоморфизме
$\pi$ (здесь,
разумеется,
$\mathcal{B}\left(  \mathcal{X}\right)  $ и $\mathcal{K}\left(
\mathcal{X}\right)  $ --- алгебра
всех
ограниченных
операторов и
идеал всех
компактных
операторов в
банаховом
пространстве
$\mathcal{X}$). Мы
назвали (\ref{gbwf})
\textit{обобщённой
формулой
Бергера-Вонга}
(обобщенной
БВ-формулой)
в \cite{ShT2002}, где она
была
доказана для
операторов в
произвольном
рефлексивном
банаховом
пространстве,
но вскоре
выяснилось,
что есть
несколько
видов
формул,
обобщающих
(\ref{bw}), и нужно
уточнять
значение
термина.
Отметим, что
из (\ref{gbwf}) следует
справедливость
(\ref{bw}) для
предкомпактных
семейств
операторов
вида $\lambda1+K$, где $\lambda
\in\mathbb{C}$, а $K$ ---
компактный
оператор,
что было
основой
данного в \cite{ShT2000}
доказательства
того, что
всякая
алгебра Ли
вольтерровых
операторов
имеет
инвариантное
подпространство.

В этой
работе
большую роль
играет
теория
топологических
радикалов
нормированных
алгебр,
основы
которой были
заложены в
работе
Диксона \cite{Dix} и
которая
получила
дальнейшее
развитие в \cite{rad1,
rad2}. Удобство
одновременного
использования
алгебраического
и
операторного
подходов
объясняется
тем, что
первый
позволяет
вести
редукцию по
идеалам и
фактор-алгебрам
(где
работают
радикалы), а
второй --- по
инвариантным
подпространствам
и
фактор-пространствам.
Результаты
работы также
отражают
чередование
этих
подходов. Мы
сначала
получаем
"\textit{смешанную
БВ-формулу}",
выражающую
ССР
семейства
элементов
нормированной
алгебры
через
спектральные
характеристики
ассоциированного
семейства
операторов
двустороннего
умножения.
Из неё мы
выводим (\ref{gbwf}),
которую
теперь
называем
\textit{операторной
БВ-формулой}.

В свою
очередь, (\ref{gbwf})
используется
для
получения
варианта
формулы
вычисления
ССР,
пригодного
для
произвольных
нормированных
алгебр. Это ---
\textit{алгебраическая
БВ-фор%
\-%
му%
\-%
ла с
гипокомпактным
радикалом}:
\begin{equation}
\rho(M)=\max\{\rho(M/\mathcal{R}_{\mathrm{hc}}(A)),r(M)\}. \label{aibwf}%
\end{equation}
Заметим, что
$\mathcal{R}_{\mathrm{hc}}(A)\supset\mathcal{K}(\mathcal{X})$
при $A=\mathcal{B}(\mathcal{X})$ и
включение
может быть
строгим ---
например,
при $\mathcal{X}=l^{1}$ и в
ряде других
примеров все
слабо
компактные
операторы
входят в $\mathcal{R}_{\mathrm{hc}%
}(\mathcal{B}(\mathcal{X}))$ (см.
раздел \ref{ss43}).
Таким
образом, (\ref{aibwf})
не только
обобщает
операторную
БВ-формулу (\ref{gbwf}),
но и
усиливает её.

Формулы (\ref{gbwf}) и
(\ref{aibwf}) были
анонсированы
в \cite{ShT2001}, но
публикация
доказательства
долгое время
откладывалась
по ряду
причин. Лишь
в 2008 году
авторы
поместили
доказательство
упомянутых
формул в
электронный
архив \cite{formula}, а
несколько
месяцев
спустя
Моррис,
также в
электронном
препринте \cite{Mor},
нашёл другое
доказательство
формулы (\ref{gbwf}),
основанное
на
мультипликативной
эргодической
теории.

Далее мы
рассматриваем
перспективы
получения
"оптимальной"
формулы для
вычисления
ССР.
Оказалось,
что это
можно
сделать в
рамках
теории
радикалов.
Мы покажем,
что
существует
наибольший
из
радикалов,
которыми
можно
заменить $\mathcal{R}%
_{\mathrm{hc}}$ в (\ref{aibwf}); мы
обозначаем
его $\mathcal{R}_{\mathrm{bw}}$. Мы
строим также
новый
радикал $\mathcal{R}_{\mathrm{cq}%
}^{a}$,
мажорируемый
$\mathcal{R}_{\mathrm{bw}}$. Как
следствие,
сочетание
радикала $\mathcal{R}%
_{\mathrm{cq}}^{a}$ с
гипокомпактным
радикалом
позволяет
получить
более
сильную
алгебраическую
БВ-формулу,
чем (\ref{aibwf}). В
конце работы
мы указываем
приложение
полученных
формул к
вопросу о
непрерывности
ССР.

\section{Некоторые
вспомогательные
результаты}

Алгебры,
рассматриваемые
в работе,
предполагаются
комплексными,
ассоциативными
и
нормированными.
Обозначим
через $A^{1}$
алгебру,
получаемую
из алгебры $A$
присоединением
единицы; при
этом
полагаем $A^{1}=A$,
если $A$
унитальна.
Термин \textit{
идеал}
означает
двусторонний
идеал,
\textit{оператор } ---
ограниченный
линейный
оператор в
комплексном
нормированном
пространстве,
\textit{фактор}
алгебры ---
фактор-алгебра
алгебры по
её
замкнутому
идеалу. Если $J$
--- замкнутый
идеал в $A$, то $q_{J}$ ---
это
стандартный
эпиморфизм $A$
на $A/J$; вместо $q_{J}(a)$
удобнее
иногда
писать $a/J$.
Точно так же,
образ
множества $M\subset A$
относительно
$q_{J}$
обозначается
либо $q_{J}(M)$, либо,
более
коротко, $M/J$.

Замкнутая
линейная
оболочка
множества $M$ в
нормированном
пространстве
$\mathcal{Y}$
обозначается
через $\operatorname*{span}M$
или, если
требуется
уточнение,
через $\operatorname*{span}_{\mathcal{Y}}M$,
замыкание ---
через $\overline{M}$ или
$\overline{M}^{\mathcal{Y}}$,
замкнутый
единичный
шар в $\mathcal{Y}$ ---
через $\mathcal{Y}_{\odot}$.

Говоря об
основных
конструкциях
работы,
отметим иные
характеризации
BW-радиуса $r(M)$
ограниченного
подмножества
$M $ алгебры $A$.
Поскольку верно $\rho\left(
a^{n}\right)  =\rho\left(  a\right)  ^{n}$, ясно,
что $r_{n}\left(  M\right)  =r_{1}\left(  M^{n}\right)
^{1/n}\leq r_{nm}\left(  M\right)  $ для
любых $n,m\in\mathbb{N}$, где
$r_{n}(M)$
определено в
(\ref{nach2}), откуда
\begin{equation}
r(M)=\sup_{n}r_{n}\left(  M\right)  =\sup_{n}r_{nm}\left(  M\right)
=\lim_{k\rightarrow\infty}r_{\mathbf{n}\left(  k\right)  }\left(  M\right)
\label{nach1}%
\end{equation}
для
некоторой
фиксированной
последовательности
$\mathbf{n}\left(  k\right)  $:
например,
можно взять
$\mathbf{n}\left(  k\right)  =\left(  p_{1}\cdots p_{k}\right)  ^{k}$,
где $p_{i}$
пробегает
ряд всех
простых чисел.

Переход от
алгебраической
ситуации к
операторной
удобно
осуществлять
через
рассмотрение
операторов
умножения.
Как обычно,
символами $\mathrm{L}%
_{a}$ и $\mathrm{R}_{a}$ мы
обозначаем
операторы
левого и
правого
умножения на
элемент $a$ в
алгебре $A$: $\mathrm{L}_{a}x=ax$,
$\mathrm{R}_{a}x=xa.$ Далее,
для $M\subset A$
полагаем $\mathrm{L}%
_{M}=\{\mathrm{L}_{a}:a\in M\}$ и $\mathrm{R}_{M}=\{\mathrm{R}_{a}:a\in
M\}.$ Нетрудно
проверить,
что если $M$
ограничено,
то $r(\mathrm{L}_{M})=r(\mathrm{R}_{M})=r(M)$ и
$\rho(\mathrm{L}_{M})=\rho(\mathrm{R}_{M})=\rho(M)$.
Важно, что $r(M)$ и
$\rho(M)$ находят
отражение в
свойствах
семейства
операторов
$\mathrm{L}_{M}\mathrm{R}_{M}=\{\mathrm{L}_{a}\mathrm{R}_{b}:a,b\in M\}.$

\begin{lemma}
\label{pass}Пусть $M$ ---
ограниченное
подмножество
нормированной
алгебры $A$.
Тогда $\rho(M^{m})=\rho(M)^{m}$ и
$r(M^{m})=r(M)^{m}$ для
любого $m\in\mathbb{N}$, $\rho(M)^{2}%
=\rho(\mathrm{L}_{M}\mathrm{R}_{M})$ и $r(M)^{2}=r(\mathrm{L}%
_{M}\mathrm{R}_{M})$.
\end{lemma}

\begin{proof}
Учитывая (\ref{nach}),
(\ref{nach1}) и то, что $\left(
M^{m}\right)  ^{n}=M^{mn}$ для
любого $n\in\mathbb{N}$,
получим
\begin{align*}
\rho(M^{m}) &  =\lim_{n\rightarrow\infty}\left\Vert M^{mn}\right\Vert
^{1/n}=\lim_{n\rightarrow\infty}\left(  \left\Vert M^{mn}\right\Vert
^{1/mn}\right)  ^{m}=\rho(M)^{m},\\
r\left(  M^{m}\right)   &  =\sup_{n}r_{n}\left(  M^{m}\right)  =\sup_{n}%
r_{1}\left(  M^{mn}\right)  ^{1/n}=\sup_{n}r_{nm}\left(  M\right)
^{m}=r\left(  M\right)  ^{m}.
\end{align*}

Так как $\Vert\left(  \mathrm{L}%
_{M}\mathrm{R}_{N}\right)  ^{n}\Vert\leq\Vert M^{n}\Vert\Vert N^{n}\Vert$,
то $\rho(\mathrm{L}_{M}\mathrm{R}_{M})\leq\rho(M)^{2}$.
Чтобы
доказать
обратное
неравенство,
заметим, что
$\Vert M^{3}\Vert\leq\Vert\mathrm{L}_{M}\mathrm{R}_{M}\Vert\Vert M\Vert.$
Меняя здесь $M$
на $M^{n}$,
извлекая
корни и
переходя к
пределу в
обеих частях
неравенства,
получим $\rho(M)^{3}=\rho(M^{3}%
)\leq\rho(\mathrm{L}_{M}\mathrm{R}_{M})\rho(M).$

Так как $\mathrm{L}_{M}$
коммутирует
с $\mathrm{R}_{M}$, то $\rho(\mathrm{L}_{a}\mathrm{R}%
_{b})\leq\rho(\mathrm{L}_{a})\rho(\mathrm{R}_{b})\leq\rho(a)\rho(b)$
для любых $a,b\in M^{n}$.
Учитывая (\ref{nach2}),
беря точные
верхние
грани,
извлекая
корни и
переходя к
верхнему
пределу,
получим $r(\mathrm{L}_{M}%
\mathrm{R}_{M})\leq r(M)^{2}.$

Так как $\rho(a)^{2}=\rho(\mathrm{L}%
_{a}\mathrm{R}_{a})$, то $\sup_{a\in M^{n}}\rho(a)^{2}\leq
\sup_{b,c\in M^{n}}\rho(\mathrm{L}_{b}\mathrm{R}_{c}).$
Снова
извлекая
корни и
переходя к
верхнему
пределу,
получим $r(M)^{2}\leq r(\mathrm{L}%
_{M}\mathrm{R}_{M}).$
\end{proof}

Запишем
определение
существенного
ССР $\rho_{e}$ в более
подробной
форме:
\[
\rho_{e}(M):=\underset{n\rightarrow\infty}{\lim}\left(  \sup\left\{
\left\Vert T\right\Vert _{e}:T\in M^{n}\right\}  \right)  ^{1/n}=\underset
{n}{\inf}\left(  \sup\left\{  \left\Vert T\right\Vert _{e}:T\in M^{n}\right\}
\right)  ^{1/n},
\]
где $||T||_{e}=||\pi(T)||$ ---
существенная
норма $T$.
Характеристикой
оператора,
близкой к $||T||_{e}$,
является
хаусдорфова
мера
некомпактности
(\textit{хаусдорфова
норма}) $\left\Vert T\right\Vert _{\chi\text{ }%
}$ образа
единичного
шара $\mathcal{X}_{\odot}$ при
отображении
$T$ (напомним,
что это ---
точная
нижняя грань
множества
тех $\varepsilon>0$, для
которых $T\mathcal{X}_{\odot}$
имеет
конечную $\varepsilon
$-сеть, и что $\Vert
T\Vert_{\chi}\leq||T||_{e}$).
Преимущество
полунормы $\left\Vert
T\right\Vert _{\chi\text{ }}$
объясняется
тем, что она
хорошо
согласована
с сужением
оператора на
инвариантное
подпространство
и с его
индуцированным
действием в
фактор-пространстве
(см. например
\cite[лемма 2.5]{ShT2002}):
\begin{equation}
\Vert T|\mathcal{Y}\Vert_{\chi}\leq2\Vert T\Vert_{\chi}\text{ и
}\left\Vert T|\left(  \mathcal{X}/\mathcal{Y}\right)  \right\Vert _{\chi}%
\leq\left\Vert T\right\Vert _{\chi}. \label{restQuot}%
\end{equation}
Важный
технический
результат
был получен
в \cite[следствие
6.5]{ShT2000}:

\begin{lemma}
\label{ineq} $\Vert\mathrm{L}_{M}\mathrm{R}_{M}\Vert_{\chi}\leq16\Vert
M\Vert_{\chi}\Vert M\Vert$ для
любого
предкомпактного
множества $M$ операторов.
\end{lemma}

Для
соответствующего
\textit{хаусдорфова
}ССР $\rho_{\chi}$,
определенного
равенством
\[
\rho_{\chi}(M):=\underset{n\rightarrow\infty}{\lim}\left(  \sup\left\{
\left\Vert T\right\Vert _{\chi}:T\in M^{n}\right\}  \right)  ^{1/n}=\inf
_{n}\left(  \sup\left\{  \left\Vert T\right\Vert _{\chi}:T\in M^{n}\right\}
\right)  ^{1/n},
\]
доказано \cite{Mor},
что он
совпадает с
$\rho_{e}$ на
предкомпактных
множествах
операторов.
Однако мы
будем
использовать
далее лишь
очевидное
неравенство
\begin{equation}
\rho_{\chi}(M)\leq\rho_{e}(M) \label{he}%
\end{equation}
для любого
предкомпактного
множества
операторов $M$.
Из него
следует, что
для
доказательства
(\ref{gbwf})
достаточно
установить
равенство
\begin{equation}
\rho\left(  M\right)  =\max\left\{  r\left(  M\right)  ,\rho_{\chi}\left(
M\right)  \right\}  . \label{ggbwf}%
\end{equation}
Важный
случай
выполнения
этого
равенства
доказан в
\cite[предложение
9.6]{ShT2000}:

\begin{lemma}
\label{special}Если $\rho\left(  M\right)  =1$
для
предкомпактного
множества $M$
операторов и
полугруппа,
порождённая
$M $,
ограничена,
то \emph{(\ref{ggbwf})}
выполнено
для $M$.
\end{lemma}

Для любого
ограниченного
множества $M$
элементов
нормированной
алгебры $A$
положим $\rho^{\chi}(M)=\rho_{\chi
}(\mathrm{L}_{M}\mathrm{R}_{M})^{1/2}$. Легко
видеть,
учитывая (\ref{restQuot}),
что для
любого
замкнутого
идеала $J$
алгебры
\begin{equation}
\rho^{\chi}(M/J)\leq\rho^{\chi}(M). \label{MJM}%
\end{equation}

Элемент $a\in A$
называется
\textit{компактным}%
, если
оператор $\mathrm{W}%
_{a}:=\mathrm{L}_{a}\mathrm{R}_{a}$
компактен в $A$.
Будем
говорить,
что
множество $M\subset
A$\textit{ состоит из
совместно
компактных
элементов }в
$A$\textit{,} если все
операторы из
$\mathrm{L}_{M}\mathrm{R}_{M}$
компактны в $A$.
Нам
понадобится
следующее
обобщение
основного
результата
работы \cite{Tur98}.

\begin{lemma}
\label{semig} Если $G$ ---
полугруппа
квазинильпотентных
элементов
нормированной
алгебры $A$ и $G$
состоит из
совместно
компактных
элементов в $A$,
то $\operatorname*{span}G$
состоит из
квазинильпотентных
элементов.
\end{lemma}

\begin{proof}
Ясно, что $\mathrm{L}_{G}%
\mathrm{R}_{G}$ ---
полугруппа
вольтерровых
операторов в
$A$. В силу
\cite[теорема 4]{Tur98},
$\operatorname*{span}\mathrm{L}_{G}\mathrm{R}_{G}$
также
состоит из
вольтерровых
операторов.
Так как $\mathrm{L}_{b}\mathrm{R}_{c}%
\in\operatorname*{span}\mathrm{L}_{G}\mathrm{R}_{G}$ для
любых $b,c\in\operatorname*{span}G$,
то $\mathrm{L}_{\operatorname*{span}G}\mathrm{R}%
_{\operatorname*{span}G}$ также
состоит из
квазинильпотентных
операторов,
и, значит, $\operatorname*{span}G$
состоит из
квазинильпотентных
элементов.
\end{proof}

\section{Топологические
радикалы}

\textit{Топологический
радикал} (ТР)
на классе
всех
нормированных
алгебр --- это
отображение
$P$,
сопоставляющее
каждой
алгебре $A$ её
замкнутый
идеал $P(A)$ и
удовлетворяющее
условиям:

\begin{itemize}
\item[$\left(  \mathrm{R1}\right)  $] $f(P(A))\subset P(B)$
для любого
непрерывного
открытого
эпиморфизма
$f:A\rightarrow B$.

\item[$\left(  \mathrm{R2}\right)  $] $P(A/P(A))=0$.

\item[$\left(  \mathrm{R3}\right)  $] $P(P(A))=P(A)$.

\item[$\left(  \mathrm{R4}\right)  $] Для
любого
идеала $J$
алгебры $A$, $P(J)$ ---
идеал
алгебры $A$,
содержащийся
в $P(A)$.
\end{itemize}

По сравнению
с
определением
ТР в \cite{Dix}, мы
требуем
открытости
эпиморфизма
в (R1) и не
требуем
замкнутости
идеала в (R4).
Алгебра
называется
$P$%
\textit{-радикальной}%
, если $A=P(A)$. Можно
показать (см.
\cite[следствие 2.8
и теорема 2.9(i)]{rad1}),
что
замыкания $P$%
-ра%
\-%
ди%
\-%
каль%
\-%
ных идеалов
и факторы $P$-ра%
\-%
ди%
\-%
каль%
\-%
ных алгебр
$P$-радикальны,
и что класс
всех $P$-ра%
\-%
ди%
\-%
каль%
\-%
ных алгебр
устойчив
относительно
расширений
(если $J$
замкнут, $A/J$ и $J$
$P$-радикальны,
то $A$ $P$%
-радикальна).
Более общий
результат
сформулирован
в лемме \ref{transext}.

Назовём
занумерованное
ординалами
семейство $(J_{\alpha
})_{\alpha\leq\gamma}$ идеалов
алгебры $A$,
содержащихся
и замкнутых
в некотором
идеале $J$
алгебры $A$,
\textit{возрастающей
трансфинитной
цепочкой }%
идеалов,
если $J_{\alpha}\subset J_{\beta}$
при $\alpha<\beta$ и $J_{\beta}=\overline{\cup
_{\alpha<\beta}J_{\alpha}}^{J}$ для
любого
предельного
ординала $\beta\leq\gamma$,
и
\textit{убывающей
трансфинитной
цепочкой, }%
если $J_{\beta}\subset J_{\alpha}$
при $\alpha<\beta$ и $J_{\beta}=\cap_{\alpha<\beta
}J_{\alpha}$ для любого
предельного
ординала $\beta\leq\gamma$.
Следующий
результат
выводится из
\cite[теоремы 2.9(i) и
2.10(i)]{rad1}.

\begin{lemma}
\label{transext} Пусть $P$ ---
топологический
радикал.
Если в
возрастающей
трансфинитной
цепочке $(J_{\alpha})_{\alpha
\leq\gamma}$ замкнутых
идеалов
нормированной
алгебры $A$
первый идеал
$J_{0}$ и все
факторы $J_{\alpha+1}/J_{\alpha}$
$P$-радикальны,
то и её
последний
элемент $J_{\gamma}$ $P$-радикален.
\end{lemma}

Мы, в
основном,
будем иметь
дело со
специальными
классами
радикалов,
обладающих
особенно
удобными
свойствами.
Топологический
радикал $P$
называется
\textit{наследственным}
(сокращенно,
НТР), если
выполняется условие

\begin{itemize}
\item[$\left(  \mathrm{R5}\right)  $] $P(J)=J\cap P(A)$ для
любого
идеала $J$
алгебры $A$.
\end{itemize}

%

\noindent
и
\textit{равномерным}
(см. \cite{rad1}), если
все
подалгебры
произвольной
$P$%
-радикальной
алгебры $P$%
-радикальны.
Легко
видеть, что
равномерные
радикалы
наследственны,
и что (R5) влечет
(R3) и (R4). Поэтому
НТР можно
определять
условиями (R1), (R2)
и (R5), чем мы
будем
пользоваться
в дальнейшем.

Имеет смысл
также
рассматривать
ТР в классе
всех
банаховых
алгебр,
сужая
действие
вышеперечисленных
аксиом. Так,
при этом, в (R4)
под идеалом
понимается
замкнутый
идеал, а
требование
открытости
эпиморфизма
в (R1) излишне (то
есть,
выполняется
автоматически
в силу
теоремы Банаха).

Если $P$ --- НТР в
классе
банаховых
алгебр (с,
возможно,
более
широкой
областью
определения),
то, согласно
\cite[теорема 2.21]{rad1},
отображение
$P^{r}$,
определённое
равенством
$P^{r}\left(  A\right)  :=A\cap P\left(  \widehat{A}\right)  $,
где $\widehat{A}$ ---
пополнение
нормированной
алгебры, --- HТР в
классе
нормированных
алгебр.
Поскольку $P^{r}\left(
A\right)  =P\left(  A\right)  $ для
банаховых
алгебр,
\textit{регулярная
процедура} $P\longmapsto
P^{r}$ расширяет
действие $P$ с
банаховых
алгебр на
нормированные
алгебры,
причём $P^{rr}=P^{r}$. ТР $P$
в классе
нормированных
алгебр
называется
\textit{регулярным
}(см. \cite{rad1}), если $P=P^{r}$.

Потребность
в
регуляризации
иллюстрируется
следующим
примером:
радикал
Джекобсона ---
НТР в классе
банаховых
алгебр
(обозначаем
его через $\operatorname*{Rad}%
$), но не
является ТР
в классе
нормированных
алгебр
\cite[пример 10.1]{Dix}.
Поэтому
вместо $\operatorname*{Rad}$
мы
рассматриваем
\textit{регулярный
радикал
Джекобсона
}$\operatorname*{Rad}^{r}$. Ясно,
что он
равномерен.
Другие
примеры ТР
мы
рассмотрим ниже.

В классе
всех
(топологических)
радикалов
введем
частичный
порядок,
полагая $P_{1}\leq P_{2}$,
если $P_{1}(A)\subset P_{2}(A)$ для
любой
нормированной
алгебры.

\begin{lemma}
\label{sup} \emph{a)} Любое
семейство
радикалов $\{P_{i}%
:i\in\Lambda\}$ имеет
точную
верхнюю
грань $\vee_{i}P_{i}$ и
точную
нижнюю грань
$\wedge_{i}P_{i}$ в классе
всех ТР.

\emph{b)} Пусть $P=\vee\{P_{i}:i\in\Lambda\}$.
В любой
алгебре $A$
существует
возрастающая
трансфинитная
цепочка
замкнутых
идеалов $(J_{\alpha})_{\alpha
\leq\gamma}$, такая что
$J_{0}=0$, $J_{\gamma}=P(A)$ и
каждый
фактор $J_{\alpha+1}/J_{\alpha}$
является $P_{i}$%
-радикальным
для
некоторого $i$.

\emph{с) }Если все $P_{i}$
наследственны,
то радикал
$\wedge\{P_{i}:i\in\Lambda\}$
является
наследственным
и
сопоставляет
каждой
алгебре $A$
идеал $\cap_{i}P_{i}\left(  A\right)  $.
\end{lemma}

\begin{proof}
a) Положим $\mathrm{H}\left(
A\right)  =\overline{\sum_{i}P_{i}\left(  A\right)  }$ и $\mathrm{B}%
\left(  A\right)  =\cap_{i}P_{i}\left(  A\right)  $ (это
корректно,
так как
можно
считать $\Lambda$,
если
необходимо,
\textit{множеством}
классов
эквивалентности
совпадающих
на алгебре
радикалов).
Отображения
$\mathrm{B}$ и $\mathrm{H}$
являются
\textit{верхним} и
\textit{нижним} \textit{ТР
}(см. \cite{Dix}), то есть
удовлетворяют
всем
аксиомам ТР,
кроме, быть
может, (R3) и (R2) соответственно.

Действительно,
ясно, что
условия (R1) и (R4)
выполнены
для $\mathrm{B}$ и $\mathrm{H}$. Так
как $P_{i}\left(  A\right)  $ ---
идеал в $\mathrm{H}\left(  A\right)  $,
то
\[
P_{i}\left(  A\right)  =P_{i}\left(  P_{i}\left(  A\right)  \right)  \subset
P_{i}\left(  \mathrm{H}\left(  A\right)  \right)  \subset\mathrm{H}\left(
\mathrm{H}\left(  A\right)  \right)  ,
\]
откуда $\mathrm{H}\left(  A\right)
\subset\mathrm{H}\left(  \mathrm{H}\left(  A\right)  \right)  $,
то есть, (R3)
выполнено
для $\mathrm{H}$.
Покажем
теперь, что
условие (R2)
выполнено
для $\mathrm{B}$. Пусть
$I=\mathrm{B}(A)$ и $J=q_{I}^{-1}(\mathrm{B}(A/I))$.
Так как $I\subset P_{i}(A)$
для любого $i$,
то задан
эпиморфизм
$p_{i}:A/I\rightarrow A/P_{i}(A)$,
который
является
непрерывным
и открытым.
Тогда
\[
\left(  p_{i}\circ q_{I}\right)  (J)=p_{i}\left(  \mathrm{B}(A/I)\right)
\subset p_{i}\left(  P_{i}(A/I)\right)  \subset P_{i}(A/P_{i}(A))=0.
\]
Следовательно,
$J\subset P_{i}(A)$ для
любого $i$,
откуда $J\subset I$ и $\mathrm{B}%
(A/I)=0$.

Свяжем с $\mathrm{H}$ и
$\mathrm{B}$
трансфинитные
цепочки
замкнутых
идеалов
алгебры $A$,
возрастающую
$\left(  \mathrm{H}^{\left(  \alpha\right)  }\left(  A\right)  \right)  $
и убывающую
$\left(  \mathrm{B}^{\alpha}\left(  A\right)  \right)  $,
определенные
реккурентно:
$\mathrm{H}^{\left(  \alpha+1\right)  }\left(  A\right)  $ ---
прообраз $\mathrm{H}\left(
A/\mathrm{H}^{\left(  \alpha\right)  }\left(  A\right)  \right)  $ в $A$
и $\mathrm{B}^{\alpha+1}\left(  A\right)  =\mathrm{B}\left(
\mathrm{B}^{\alpha}\left(  A\right)  \right)  $ при
начальных
условиях $\mathrm{H}^{\left(
0\right)  }=0$ и $\mathrm{B}^{0}=A$.
Отображения
$\mathrm{B}^{\alpha}$ и $\mathrm{H}^{\left(  \alpha\right)  }$ ---
соответственно
верхний и
нижний ТР. По
существу,
это
установлено
в \cite[теоремы 6.6 и
6.10]{Dix}: аргументы
\cite{Dix}
непосредственно
переносятся
с банаховых
алгебр на
нормированные,
и только
один момент
требует
пояснений ---
это условие (R4)
в случае
незамкнутого
идеала $J$. Для
$\mathrm{B}^{\alpha}$ оно
легко
доказывается
индукцией;
проверим его
для $\mathrm{H}^{\left(  \alpha\right)  }$.
Если по
предположению
индукции $I=\mathrm{H}%
^{\left(  \alpha\right)  }\left(  J\right)  $ ---
идеал в $A$,
содержащийся
в $K=\mathrm{H}^{\left(  \alpha\right)  }\left(  \overline{J}\right)  $,
то
отображение
$x/I\longmapsto q_{\overline{I}}\left(  x\right)  $
осуществляет
изоморфизм
между $J/I$ и
идеалом $q_{\overline{I}}\left(
J\right)  $ алгебры $A/\overline{I}%
$. Тогда
\begin{equation}
\mathrm{H}^{\left(  \alpha+1\right)  }\left(  J\right)  =\left\{  x\in
J:q_{\overline{I}}\left(  x\right)  \in\mathrm{H}\left(  q_{\overline{I}%
}\left(  J\right)  \right)  \right\}  \label{local}%
\end{equation}
--- идеал в $A$. Так
как $q_{\overline{I}}\left(  J\right)  $ ---
идеал в $\overline{J}/\overline{I}$,
то $\mathrm{H}\left(  q_{\overline{I}}\left(  J\right)  \right)
\subset\mathrm{H}\left(  \overline{J}/\overline{I}\right)  $.
Так как $p\left(  \mathrm{H}\left(
\overline{J}/\overline{I}\right)  \right)  $ $\subset\mathrm{H}\left(
\overline{J}/K\right)  $ для
стандартного
эпиморфизма
$p:\overline{J}/\overline{I}\longrightarrow\overline{J}/K$, то
$\mathrm{H}\left(  \overline{J}/\overline{I}\right)  \subset q_{\overline{I}%
}\left(  \mathrm{H}^{\left(  \alpha+1\right)  }\left(  \overline{J}\right)
\right)  $ и тогда $\mathrm{H}^{\left(
\alpha+1\right)  }\left(  J\right)  \subset\mathrm{H}^{\left(  \alpha
+1\right)  }\left(  \overline{J}\right)  $ в
силу (\ref{local}); шаг
же индукции
для
предельного
ординала
ясен. Итак,
показано,
что $\mathrm{H}^{\left(  \alpha\right)  }\left(  J\right)  $
--- идеал в $A$ и $\mathrm{H}%
^{\left(  \alpha\right)  }\left(  J\right)  \subset\mathrm{H}^{\left(
\alpha\right)  }\left(  \overline{J}\right)  $ для
любого
ординала $\alpha$.
Учитывая
включение $\mathrm{H}%
^{\left(  \alpha\right)  }\left(  \overline{J}\right)  \subset\mathrm{H}%
^{\left(  \alpha\right)  }\left(  A\right)  $,
доказанное в
\cite{Dix}, получим,
что $\mathrm{H}^{\left(  \alpha\right)  }$
удовлетворяет
(R4).

В силу
монотонности,
цепочки $\left(  \mathrm{B}^{\alpha
}\left(  A\right)  \right)  $ и $\left(  \mathrm{H}^{\left(
\alpha\right)  }\left(  A\right)  \right)  $
стабилизируются:
найдутся
такие
ординалы $\beta$ и
$\gamma$, что $\mathrm{H}^{\left(  \beta+1\right)  }\left(
A\right)  =\mathrm{H}^{\left(  \beta\right)  }\left(  A\right)  $ и
$\mathrm{B}^{\gamma+1}\left(  A\right)  =\mathrm{B}^{\gamma}\left(  A\right)
$. Найденные
при этом
идеалы
обозначим
через $\mathrm{H}^{\prime}\left(  A\right)  $
и $\mathrm{B}^{\prime}\left(  A\right)  $. Из
условий
стабилизации
легко
следует
согласно
\cite[теоремы 6.6 и 6.10]{Dix}
(оговорка,
сделанная
выше, здесь
уже не нужна),
что
соответствующие
отображения
$\mathrm{B}^{\prime}$ и $\mathrm{H}^{\prime}$
удовлетворяют
(R3) и (R2), то есть
становятся
радикалами.
Ясно, что $\mathrm{H}^{\prime}$
и $\mathrm{B}^{\prime}$ ---
верхняя и
нижняя грани
семейства $\{P_{i}%
:i\in\Lambda\}$. Если $T$ --- ТР,
являющийся
верхней
гранью $\{P_{i}:i\in\Lambda\}$,
то по
индукции
получим, что
$\mathrm{H}^{\left(  \alpha\right)  }\leq T^{\left(  \alpha\right)  }=T$
для любого $\alpha$,
откуда $\mathrm{H}^{\prime}\leq T$.
Значит, $\mathrm{H}^{\prime}$ ---
точная
верхняя
грань $\{P_{i}:i\in\Lambda\}$.
Аналогично
можно
показать,
что $\mathrm{B}^{\prime}$ --
точная
нижняя грань
$\{P_{i}:i\in\Lambda\}$ (см. также
\cite[следствие
6.12]{Dix}).

b) Допустим,
рассуждая по
индукции,
что мы уже
построили
требуемый
идеал $J_{\alpha}$.
Поскольку $J_{\alpha
}\subset P\left(  A \right)  $, можно
проверить,
что $P\left(  A\right)  /J_{\alpha}=P\left(  A/J_{\alpha
}\right)  $ (см.
например
\cite[доказательство
следствия 2.8
(ii)]{rad1}). Если $J_{\alpha}\neq P\left(  A
\right)  $, то $\mathrm{H}\left(  A/J_{\alpha}\right)  \neq0$,
что даёт
возможность
построить $J_{\alpha+1}$.

c) Поскольку
$\mathrm{B}$
удовлетворяет
условиям (R1) и (R2),
достаточно
доказать (R5)
для $\mathrm{B}$. Но это
очевидно в
силу того,
что (R5)
выполнено
для всех $P_{i}$.
\end{proof}

\subsection{Компактно
квазинильпотентный
радикал.}

Радикал $\mathcal{R}_{\mathrm{cq}}$,
который мы
сейчас
определим,
связан с ССР.
Пусть $\mathfrak{K}(A)$ ---
множество
всех
предкомпактных
подмножеств
нормированной
алгебры $A$.
Назовём $A$
\textit{компактно
квазинильпотентной,}
если ${\rho}(M)=0$ для
любого $M\in\mathfrak{K}(A)$.
Следующий
результат
доказан в
\cite[теоремы 4.18, 4.20 и
4.25, лемма 4.11]{rad1}.

\begin{theorem}
\label{cq-rad}

\emph{(a)} В любой
нормированной
алгебре $A$
есть
наибольший
компактно
квазинильпотентный
идеал $\mathcal{R}_{\mathrm{cq}}(A)$.
Отображение
$A\longmapsto\mathcal{R}_{\mathrm{cq}}(A)$ ---
наследственный
топологический
радикал.

\emph{(b)} Элемент $a\in A$
принадлежит
$\mathcal{R}_{\mathrm{cq}}(A)$ тогда и
только
тогда, когда
$\rho(aM)=0$ для любого
$M\in\mathfrak{K}(A)$.
Эквивалентное
условие: $\rho(\{a\}\cup M)=\rho(M)$
для любого
$M\in\mathfrak{K}(A)$.

\emph{(c)} $\rho(M)=\rho(M/\mathcal{R}_{\mathrm{cq}}(A))$
для всех $M\in\mathfrak{K}(A)$.

\emph{(d)} $\mathcal{R}_{\mathrm{cq}}\left(  B\right)  =B\cap\mathcal{R}%
_{\mathrm{cq}}\left(  A\right)  $ для
любой
плотной
подалгебры $B$.
\end{theorem}

Условие (d)
эквивалентно
регулярности
$\mathcal{R}_{\mathrm{cq}}$. Ясно,
что $\mathcal{R}_{\mathrm{cq}}$ равномерен.

\subsection{Гипокомпактный
радикал.}

Обозначим
через $\mathcal{C}(A)$
множество
всех
компактных
элементов
нормированной
алгебры $A$. Это
--- замкнутый
полугрупповой
идеал в $A$, так
что $\operatorname{span}\mathcal{C}(A)$ ---
идеал в $A$.

\begin{lemma}
\label{lhc1} Если $f:A\longrightarrow B$ ---
непрерывный
открытый
эпиморфизм
нормированных
алгебр, то $f(\mathcal{C}%
(A))\subset\mathcal{C}(B)$.
\end{lemma}

\begin{proof}
Пусть $a \in\mathcal{C}(A)$.
Так как $f\mathrm{W}_{a}=\mathrm{W}%
_{f\left(  a\right)  }f$ и $f\left(  A_{\odot}\right)  $
содержит
открытый шар
пространства
$B$, то $\mathrm{W}_{f\left(  a\right)  }$
компактен в $B$.
\end{proof}

\begin{lemma}
\label{lhc2} Пусть $J$ ---
идеал в $A$. Если
$\mathcal{C}(J)\neq0$, то $J\cap\mathcal{C}(A)\neq0$.
\end{lemma}

\begin{proof}
Легко
видеть, что
\begin{equation}
\mathrm{W}_{ba}=\mathrm{L}_{b}\mathrm{W}_{a}\mathrm{R}_{b}=\mathrm{R}%
_{a}\mathrm{W}_{b}\mathrm{L}_{a} \label{prod}%
\end{equation}
для всех $a,b\in A$.
Если $a\in\mathcal{C}(J)$, то
при любом $b\in J$
оператор $\mathrm{W}_{ba}$
компактен в $A$.
Поэтому $J\mathcal{C}(J)\subset
J\cap\mathcal{C}(A)$ и всё
доказано,
если $J\mathcal{C}(J)\neq0$. С
другой
стороны,
если $J\mathcal{C}(J)=0$, то
$\mathcal{C}(J)\subset\mathcal{C}(A)$,
поскольку $\mathrm{W}%
_{a}(x)=\left(  ax\right)  a=0$ для
всех $a\in\mathcal{C}(J)$ и $x\in A$.
\end{proof}

Следующий
результат
легко
следует из (\ref{prod}).

\begin{lemma}
\label{lhc3}Замкнутый
идеал,
порожденный
компактным
элементом
нормированной
алгебры $A$,
состоит из
совместно
компактных
элементов в $A$.
\end{lemma}

Нормированная
алгебра $A$
называется
\textit{гипокомпактной,}
если любой
её ненулевой
фактор $A/J$
содержит
ненулевой
компактный
элемент, и
\textit{бикомпактной}%
, если $A$
состоит из
совместно
компактных
элементов в $A$.
Эти понятия
распространим
на идеалы и
подалгебры,
рассматривая
их как
нормированные
алгебры. Из
определения
вытекает,
что
замыкания
гипокомпактных
подалгебр и
факторы
гипокомпактных
алгебр гипокомпактны.

\begin{lemma}
\label{iha}Любой, даже
односторонний,
ненулевой
идеал
гипокомпактной
алгебры
содержит
ненулевой
компактный
элемент этой алгебры.
\end{lemma}

\begin{proof}
Пусть $A $
гипокомпактна
и $J$ --- её
ненулевой
левый идеал.
Пусть $I=\{a\in A:aJ=0\}$.
Тогда $I$ ---
замкнутый
идеал в $A$. Если
$I=A$, то всякий
элемент в $J$ ---
компактный
элемент в $A $.
Иначе $A/I$ имеет
ненулевой
компактный
элемент $b$.
Выберем $a\in A$,
такой что $q_{I}\left(
a\right)  =b$. Очевидно,
найдется $x\in J$,
такой что $ax\neq0$.
Докажем, что
элемент $ax$
компактен в $A$.

Действительно,
$\mathrm{W}_{ax}=\mathrm{R}_{x}\mathrm{W}_{a}\mathrm{L}_{x}$.
Так как $\mathrm{R}_{x}(I)=0$,
определим
оператор $V:A/I\longrightarrow
A$, полагая $V\left(  q_{I}\left(
y\right)  \right)  =yx$ для
любого $y \in A$.
Тогда
нетрудно
видеть, что $\left\Vert
V\right\Vert \leq\left\Vert \mathrm{R}_{x}\right\Vert $ и $\mathrm{R}%
_{x}=Vq_{I}$. Поскольку
оператор $q_{I}\mathrm{W}%
_{a}=\mathrm{W}_{b}q_{I}$
компактен,
оператор $\mathrm{W}%
_{ax}=Vq_{I}\mathrm{W}_{a}\mathrm{L}_{x}$
компактен в $A$.
\end{proof}

Ясно, что
всякая
бикомпактная
алгебра
гипокомпактна.
Мы установим
в следующем
результате,
что все
гипокомпактные
алгебры
могут быть
получены
последовательным
расширением бикомпактных.

\begin{proposition}
\label{phc1} Пусть $J$ ---
идеал в
нормированной
алгебре $A$.
Следующие
условия эквивалентны:

\begin{itemize}
\item[(i)] $J$ гипокомпактен.

\item[(ii)] Для
всякого
непрерывного
открытого
эпиморфизма
$f:A\longrightarrow B$ либо $f(J)=0$,
либо $f(J)\cap\mathcal{C}(B)\neq0$.

\item[(iii)]
Существует
возрастающая
трансфинитная
цепочка $(J_{\alpha})_{\alpha
\leq\gamma}$ идеалов в $A$,
такая что $J_{0}=0$,
$J_{\gamma}=J$, все $J_{\alpha}$
замкнуты в $J$ и
все факторы
$J_{\alpha+1}/J_{\alpha}$ бикомпактны.
\end{itemize}
\end{proposition}

\begin{proof}
(i)$\Longrightarrow$(ii) Пусть $I=\ker f$
и $K=I\cap\overline{J}$. Тогда
алгебра $\overline{J}$ и
её фактор $\overline{J}/K$ ---
гипокомпактные
алгебры, в то
время как $q_{K}\left(
J\right)  $ ---
одновременно
идеал в $A/K$ и $\overline{J}/K$.
Предположим,
что $f(J)\neq0$. Тогда
$q_{K}\left(  J\right)  \neq0$ и $q_{K}\left(  J\right)  \cap
\mathcal{C}(\overline{J}/K)$ отлично
от нуля по
лемме \ref{iha}. Так
как это
множество
содержится в
$\mathcal{C}(q_{K}\left(  J\right)  )$, то
найдётся
ненулевой
элемент $a\in q_{K}\left(  J\right)
\cap\mathcal{C}(A/K)$ по лемме
\ref{lhc2}. Определим
отображение
$g:A/K\longrightarrow B$
условием $g\left(
q_{K}\left(  b\right)  \right)  =f\left(  b\right)  $ для
любого $b\in A$.
Ясно, что $g$ ---
непрерывный
открытый
эпиморфизм.
Выберем
такой
элемент $b\in J$,
что $a=q_{K}\left(  b\right)  $.
Тогда $f\left(  b\right)  =g\left(  a\right)  $
--- ненулевой
компактный
элемент
алгебры $B$ по
лемме \ref{lhc1}.

(ii)$\Longrightarrow$(iii)
Рассмотрим
все
возрастающие
трансфинитные
цепочки $(J_{\alpha})_{\alpha
\leq\beta}$ идеалов в $A$,
такие что $J_{\alpha}\subset
J$ и $J_{\alpha}$ замкнут
в $J$, $J_{\alpha+1}/J_{\alpha}$
бикомпактен
и отличен от
нуля для
любого $\alpha<\beta$.
Такие
цепочки
образуют
множество (а
не класс),
поскольку $A$ ---
множество.
Упорядочим
его
условием
\[
(J_{\alpha})_{\alpha\leq\beta_{1}}\prec(I_{\alpha})_{\alpha\leq\beta_{2}%
}\text{, если }\beta_{1}\leq\beta_{2}\text{ и
}J_{\alpha}=I_{\alpha}\text{ при }\alpha\leq\beta_{1}.
\]
По лемме
Цорна, в этом
множестве
есть
максимальный
элемент $(J_{\alpha})_{\alpha
\leq\gamma}$. Пусть $K=\overline{J_{\gamma}}$.
Если $J_{\gamma}\neq J$, то $K\neq
\overline{J}$, и фактор
$J/J_{\gamma}$
изометрически
изоморфен $q_{K}\left(
J\right)  $. Полагая $B=A/K$ и
$f=q_{K}$, из (ii)
выводим, что
$f(J)\cap\mathcal{C}(B)\neq0$. По
лемме \ref{lhc3},
существует
ненулевой
замкнутый
бикомпактный
идеал $I$ в $B$,
порождённый
элементом из
$f(J)$.

Положим $J_{\gamma+1}=\{x\in
J:f(x)\in I\}.$ Тогда
легко
видеть, что
$J_{\gamma+1}\neq J_{\gamma}$, $J_{\gamma+1}$ ---
идеал в $A $, $J_{\gamma+1}$
замкнут в $J$ и
$J_{\gamma+1}/J_{\gamma}$
бикомпактен.
Следовательно,
к цепочке
можно
добавить $J_{\gamma+1}$,
вопреки
предположению
максимальности.

(iii)$\Longrightarrow$(i) Пусть $I$ ---
собственный
замкнутый
идеал в $J$.
Возьмём
первый
ординал ${\alpha}$,
при котором
$J_{\alpha}$ не лежит в $I$.
Тогда $0\neq q_{I}(J_{\alpha})\subset
\mathcal{C}(J/I)$.
\end{proof}

\begin{corollary}
\label{3s}\textit{Пусть }$A$ \textit{ ---
нормированная
алгебра.
Следующие
условия
эквивалентны:}%

\begin{enumerate}
\item[(i)] $A$\textit{
гипокомпактна;}%

\item[(ii)] \textit{все
идеалы и
факторы
алгебры} $A$
\textit{гипокомпактны;}%

\item[(iii)] $J$\textit{ и }$A/J$\textit{
гипокомпактны
для
некоторого
замкнутого
идеала} $J$\textit{ в }$A$\textit{.}
\end{enumerate}
\end{corollary}

\begin{proof}
(i)$\Longrightarrow$(ii) Мы уже
знаем, что
фактор
гипокомпактной
алгебры
гипокомпактен.
Пусть теперь
алгебра $A$
гипокомпактна
и $J$ --- её идеал.
Пусть $f:A\longrightarrow B$ ---
непрерывный
открытый
эпиморфизм и
$I=\ker~f$. Считая, что
$f(J)\neq0$, получим,
что $q_{I}\left(  J\right)  \neq0$. По
лемме \ref{iha},$\mathcal{\ }$
найдется
такой $a\in J$, что $0\neq
q_{I}(a)\in\mathcal{C}(A/I).$ По
лемме \ref{lhc1}, $f(a)$ ---
ненулевой
компактный
элемент
алгебры $B$.
Согласно
предложению
\ref{phc1}, $J$ гипокомпактен.

(ii)$\Longrightarrow$(iii) очевидно.

(iii)$\Longrightarrow$(i) Пусть $I$ ---
собственный
замкнутый
идеал в $A$. Если
$J\subset I$, то $A/I$ можно
отождествить
с алгеброй
$(A/J)/(I/J)$,
являющейся
фактором
гипокомпактной
алгебры. По
определению,
она содержит
компактный элемент.

Пусть теперь
$I$ не содержит
$J$. Полагая $K=J\cap I$,
имеем $\mathcal{C}(J/K)\neq0.$ По
лемме \ref{lhc2}, $J/K\cap\mathcal{C}%
(A/K)\neq0.$ Пусть $0\neq q_{K}(a)\in
J/K\cap\mathcal{C}(A/K)$. Тогда $a\notin I$
и $q_{I}(a)\in\mathcal{C}(A/I)$ по
лемме \ref{lhc1}.
Таким
образом, $A/I$
содержит
ненулевые
компактные элементы.
\end{proof}

\begin{corollary}
\label{lhclar} В любой
нормированной
алгебре есть
наибольший
гипокомпактный
идеал.
\end{corollary}

\begin{proof}
Пусть $J$ ---
замкнутая
линейная
оболочка
объединения
всех
гипокомпактных
идеалов
алгебры $A$.
Докажем, что
идеал $J$ гипокомпактен.

Согласно
предложению
\ref{phc1},
достаточно
показать,
что если $f:A\longrightarrow B$ ---
непрерывный
открытый
эпиморфизм,
причем $f(J)\neq0$, то
$f(J)\cap\mathcal{C}(B)\neq0.$ Но если
$f(J)\cap\mathcal{C}(B)=0$, то $f(I)=0$ для
любого
гипокомпактного
идеала $I$
алгебры $A$.
Следовательно,
$f(J)=0$, противоречие.
\end{proof}

Обозначим
наибольший
гипокомпактный
идеал
алгебры $A$
через $\mathcal{R}_{\mathrm{hc}}(A)$.

\begin{lemma}
\label{lhcint} Если $J$ ---
идеал в $A$, то $\mathcal{R}%
_{\mathrm{hc}}(J)=J\cap\mathcal{R}_{\mathrm{hc}}(A)$.
\end{lemma}

\begin{proof}
По следствию
\ref{3s}, идеал $J\cap\mathcal{R}_{\mathrm{hc}%
}(A)$ алгебры $J$
гипокомпактен,
и потому
содержится в
$\mathcal{R}_{\mathrm{hc}}(J)$.
Требуется
доказать
обратное включение.

Пусть $I=\operatorname*{span}_{J}\left(
A^{1}\mathcal{R}_{\mathrm{hc}}(J)A^{1}\right)  $.
Ясно, что $I$ ---
идеал в $A$. Так
как
\begin{align*}
\operatorname{span}\nolimits_{J}I^{3}  &  \subset\operatorname{span}%
\nolimits_{J}\left(  \left(  A^{1}\mathcal{R}_{\mathrm{hc}}(J)A^{1}%
A^{1}\right)  \mathcal{R}_{\mathrm{hc}}(J)\left(  A^{1}A^{1}\mathcal{R}%
_{\mathrm{hc}}(J)A^{1}\right)  \right) \\
&  \subset\operatorname{span}\nolimits_{J}\left(  J\mathcal{R}_{\mathrm{hc}%
}(J)J\right)  \subset\mathcal{R}_{\mathrm{hc}}(J),
\end{align*}
то $\operatorname*{span}_{J}I^{3}$ ---
гипокомпактный
идеал в $J$. Но
алгебра $I/\left(
\operatorname*{span}_{J}I^{3}\right)  $
бикомпактна,
поскольку $\mathrm{L}%
_{a}\mathrm{R}_{b}=0$ для
любых $a,b\in I/\left(  \operatorname*{span}%
_{J}I^{3}\right)  $. По
следствию \ref{3s},
идеал $I$
гипокомпактен.
Следовательно,
$\mathcal{R}_{\mathrm{hc}}(J)\subset I\subset\mathcal{R}_{\mathrm{hc}}(A).$
\end{proof}

\begin{lemma}
\label{lhccel} Алгебра
$A/\mathcal{R}_{\mathrm{hc}}(A)$ не
имеет
ненулевых
гипокомпактных
идеалов и
компактных элементов.
\end{lemma}

\begin{proof}
Если $J$ ---
гипокомпактный
идеал
алгебры $A/\mathcal{R}_{\mathrm{hc}%
}(A)$, то, по
следствию \ref{3s},
его прообраз
$\left\{  x\in A:q_{\mathcal{R}_{\mathrm{hc}}(A)}(x)\in J\right\}  $ ---
гипокомпактный
идеал в $A$,
строго
содержащий
$\mathcal{R}_{\mathrm{hc}}(A)$, противоречие.

По лемме \ref{lhc3},
если $A/\mathcal{R}_{\mathrm{hc}}(A)$
имеет
ненулевые
компактные
элементы, то
у нее есть
ненулевые
бикомпактные
идеалы, что невозможно.
\end{proof}

\begin{theorem}
\label{thcrad}
Отображение
$A\longmapsto\mathcal{R}_{\mathrm{hc}}(A)$ ---
наследственный
топологический
радикал.
\end{theorem}

\begin{proof}
(R2) и (R5) доказаны
в леммах \ref{lhccel} и
\ref{lhcint}; докажем (R1).

Пусть $f:A\rightarrow B$ ---
непрерывный
открытый
эпиморфизм.
Обозначим,
для
краткости,
$q_{\mathcal{R}_{\mathrm{hc}}(B)}$ через $q$.
Ясно, что $q{\circ}f$ ---
непрерывный
открытый
эпиморфизм $A$
на $B/\mathcal{R}_{\mathrm{hc}}(B)$. Так
как идеал $\mathcal{R}%
_{\mathrm{hc}}(A)$
гипокомпактен,
то, по
предложению
\ref{phc1}, её образ $\left(
q{\circ}f\right)  (\mathcal{R}_{\mathrm{hc}}(A))$ либо
равен нулю,
либо
содержит
ненулевой
компактный
элемент
алгебры $B/\mathcal{R}_{\mathrm{hc}%
}(B)$. Но
последнее
невозможно в
силу леммы \ref{lhccel}.
Значит, $\left(  q{\circ}f\right)
(\mathcal{R}_{\mathrm{hc}}(A))=0$ и $f(\mathcal{R}_{\mathrm{hc}%
}(A))\subset\mathcal{R}_{\mathrm{hc}}(B).$
\end{proof}

Таким
образом,
\textit{регулярный
гипокомпактный
радикал} $\mathcal{R}_{\mathrm{hc}%
}^{r}$, получаемый
регулярной
процедурой, ---
НТР.
Очевидно,
что $\mathcal{R}_{\mathrm{hc}}\leq\mathcal{R}_{\mathrm{hc}%
}^{r}$.

\subsection{Радикал
$\operatorname*{Rad}^{r}\wedge\mathcal{R}_{\mathrm{hc}}$.}

Для
дальнейшего
важна связь
между
радикалами
$\mathcal{R}_{\mathrm{hc}}\wedge\operatorname*{Rad}^{r}$ и
$\mathcal{R}_{\mathrm{cq}}$.

\begin{theorem}
\label{order} Радикал
$\operatorname*{Rad}^{r}\wedge\mathcal{R}_{\mathrm{hc}}$
является
наследственным,
сопоставляет
каждой
алгебре $A$
идеал $\mathcal{R}_{\mathrm{hc}}\left(
A\right)  \cap\operatorname*{Rad}^{r}\left(  A\right)  $ и
удовлетворяет
неравенствам
$\mathcal{R}_{\mathrm{hc}}\wedge\operatorname*{Rad}^{r}\leq\mathcal{R}%
_{\mathrm{cq}}\leq\operatorname*{Rad}^{r}$.
\end{theorem}

\begin{proof}
Так как
радикалы $\mathcal{R}%
_{\mathrm{hc}}$ и $\operatorname*{Rad}^{r}$
наследственны,
то $\mathcal{R}_{\mathrm{hc}}\wedge\operatorname*{Rad}^{r}$ ---
наследственный
радикал и
сопоставляет
каждой
алгебре $A$
замкнутый
идеал $\mathcal{R}_{\mathrm{hc}}\left(
A\right)  \cap\operatorname*{Rad}^{r}\left(  A\right)  $ по
лемме \ref{sup}.

Докажем
сначала, что
каждая $\operatorname*{Rad}^{r}%
$%
-радикальная
бикомпактная
алгебра $A$
компактно
квазинильпотентна.
В самом деле,
если $M\in\mathfrak{K}(A)$, то $r(M)=0$,
поскольку $A$
состоит из
квазинильпотентных
элементов.
Пусть $N=\mathrm{L}_{M}\mathrm{R}_{M}$.
Тогда $r(N)=0$ по
лемме \ref{pass}. Так
как $N$ ---
предкомпактное
множество
компактных
операторов,
то $\rho(N)=0$ по (\ref{bw}) и $\rho(M)=0$
по лемме \ref{pass}.

Пусть $A$ ---
произвольная
алгебра и
$J=\operatorname*{Rad}^{r}(A)\cap\mathcal{R}_{\mathrm{hc}}(A)$.
Так как
идеал $J$
гипокомпактен,
то, по
предложению
\ref{phc1},
существует
возрастающая
трансфинитная
цепочка $(J_{\alpha})_{\alpha
\leq\gamma}$ замкнутых
идеалов в $A$,
такая что $J_{0}=0$,
$J_{\gamma}=J$, и все
факторы $J_{\alpha+1}/J_{\alpha}$
бикомпактны.
Из
доказанного
выше
следует, что
эти факторы
$\mathcal{R}_{\mathrm{cq}}$%
-радикальны.
Тогда идеал $J$
также $\mathcal{R}_{\mathrm{cq}}$%
-радикален
по лемме \ref{transext}.

Из теоремы
\ref{cq-rad}(d) вытекает,
что
замыкание
компактно
квазинильпотентного
идеала
алгебры в её
пополнении
состоит из
квазинильпотентных
элементов и,
следовательно,
лежит в
радикале
Джекобсона
пополнения.
Отсюда
вытекает
неравенство
$\mathcal{R}_{\mathrm{cq}}\leq\operatorname*{Rad}^{r}$.
\end{proof}

В частности,
мы доказали,
что на
гипокомпактных
алгебрах
регулярный
радикал
Джекобсона
совпадает с
компактно
квазинильпотентным
радикалом.

\section{Основные
БВ-формулы}

\subsection{Смешанная
БВ-формула.}

Здесь мы
докажем, что
\begin{equation}
\rho(M)=\max\{\rho^{\chi}(M),r(M)\} \label{e2}%
\end{equation}
для любого
$M\in\mathfrak{K}(A)$. Это
равенство
достаточно
установить,
считая, что $A$ ---
банахова
алгебра и
что $M$
порождает $A$. В
самом деле,
все входящие
в (\ref{e2}) величины
не меняются
при замене
алгебры её
пополнением.
Кроме того, $\rho(M)$
и $r(M)$ не
изменятся,
если их
вычислять в
замкнутой
подалгебре
$\mathcal{A}(M)$,
порожденной
$M$. Значение же
$\rho^{\chi}(M)=\rho_{\chi}(\mathrm{L}_{M}\mathrm{R}_{M})^{1/2}$
при этом не
может
возрасти,
так как
операторы $\mathrm{L}%
_{M}\mathrm{R}_{M}$ сужаются
на
подпространство
$\mathcal{A}(M)$, а
нетривиальной
частью (\ref{e2})
является
лишь
неравенство
$\leq$.

Полугруппа $G$
элементов
алгебры $A$
называется
\textit{полугруппой
Раджави} (или,
коротко, $R$%
\textit{-полугруппой}%
), если ${\lambda}a\in G$ для
любых $a\in G$ и $\lambda\geq0$.
Пусть $M\subset A$, $\mathcal{S}(M)$ ---
полугруппа,
порожденная
$M$, и $\mathcal{S}_{+}\left(  M\right)  $ --- $R$%
-полугруппа,
порожденная
$M$. Ясно, что $\mathcal{S}%
(M)=\cup_{n=1}^{\infty}M^{n}$ и $\mathcal{S}_{+}\left(  M\right)
=\mathbb{R}_{+}\mathcal{S}(M),$ где $\mathbb{R}_{+}%
=\{t\in\mathbb{R}:t\geq0\}$.

Пусть $N$ --
множество
операторов в
банаховом
пространстве
и $G=\mathcal{S}(N)$.
Оператор $T\in N^{n}$
называется
\textit{ведущим}
(точнее, $n$%
-\textit{ведущим}),
если $\left\Vert T\right\Vert \geq\left\Vert
S\right\Vert $ для всех $S\in
\cup_{k<n}N^{k}$.
Уточнение
термина
существенно,
так как
оператор
может быть в
разных $N^{n}$.
Далее,
\textit{ведущая
последовательность}
в $G$ --- это
последовательность
$n(k)$-ведущих
операторов
$T_{k}\in N^{n\left(  k\right)  }$, такая
что $n(k)\rightarrow\infty$ и $\left\Vert
T_{k}\right\Vert \rightarrow\infty$ при $k\rightarrow\infty
$. Если $G$
неограничена,
то очевидно,
что хотя бы
одна ведущая
последовательность
в $G$ существует.

\begin{lemma}
\label{L3} Пусть $N$ ---
предкомпактное
множество
операторов.
Если верно $\rho_{\chi}(N)<\rho(N)=1$ и
полугруппа
$\mathcal{S}(N)$
неограничена,
то
существует
последовательность
операторов
единичной
нормы $T_{n}\in\mathcal{S}_{+}(N)$,
сходящаяся к
компактному
оператору $T$.
Более того,
можно
выбрать в
качестве
такой
последовательности
любую
сходящуюся
подпоследовательность
из $S_{n}/\Vert S_{n}\Vert$, где $\left\{
S_{n}\right\}  $ ---
произвольная
ведущая
последовательность
в $\mathcal{S}(N)$.
\end{lemma}

\begin{proof}
Последовательность
$\left\{  S_{n}/\Vert S_{n}\Vert\right\}  $
предкомпактна
согласно
\cite[следствие
6.8(iii)]{ShT2000}; пусть $T$ --- её
предельная
точка. В силу
$\rho_{\chi}(N)<1$,
последовательность
$\left\{  \left\Vert S_{n}\right\Vert _{\chi}\right\}  $
ограничена и
тогда
последовательность
хаудорфовых
норм
операторов
$S_{n}/\Vert S_{n}\Vert$
стремится к
нулю.
Следовательно,
$T$ компактен.
\end{proof}

\begin{lemma}
\label{L4} Пусть $A=\mathcal{A}\left(  M\right)
$, где $M$
предкомпактно,
и пусть $N=\mathrm{L}_{M}\mathrm{R}_{M}$.
Если $\rho_{\chi}(N)<\rho(N)=1$ и
полугруппа
$\mathcal{S}(N)$
неограничена,
то $\overline{\mathcal{S}_{+}(N)}$
содержит
ненулевой
компактный
оператор $T$,
такой что

\begin{itemize}
\item[(i)] оператор
$\mathrm{L}_{Th}\mathrm{R}_{Tg}$
компактен
при любых $h,g\in A$;

\item[(ii)] если $r(N)<1$, то
$T(A)\subset\operatorname*{Rad}(A)$.
\end{itemize}
\end{lemma}

\begin{proof}
Все элементы
полугруппы
$\mathcal{S}_{+}(N)$ можно
записывать в
виде $F={\lambda}\mathrm{L}_{a}\mathrm{R}_{b}$,
где $a,b\in\mathcal{S}(M)$ и ${\lambda\geq}0$.
Для
краткости,
мы полагаем
$F^{\leftrightarrows}={\lambda}\mathrm{L}_{b}\mathrm{R}_{a}$. В
выборе $F^{\leftrightarrows}$
может
присутствовать
неоднозначность,
но
независимо
от этого
выбора, для
любых $h,g\in M$
выполняется
равенство $\mathrm{L}%
_{Fh}\mathrm{R}_{Fg}=F\mathrm{L}_{h}\mathrm{R}_{g}F^{\leftrightarrows}.$

Пусть $\left\{  S_{n}\right\}  $ ---
ведущая
последовательность
в $\mathcal{S}(N)$. Для
каждого её
элемента $S_{n}\in N^{m\left(
n\right)  }$, оператор
$S_{n}^{\leftrightarrows}$ тоже
можно
выбрать в $N^{m\left(
n\right)  }$, поэтому мы
можем
дополнительно
предполагать,
что
\begin{equation}
\Vert S_{n}\Vert\geq\Vert S_{n}^{\leftrightarrows}\Vert\label{e5}%
\end{equation}
для всех $n$. По
лемме \ref{L3}, есть
последовательность
операторов
$T_{n}=S_{k_{n}}/\Vert S_{k_{n}}\Vert$,
стремящаяся
к
компактному
оператору $T$.
Все
операторы $T_{n}%
^{\leftrightarrows}=S_{k_{n}}^{\leftrightarrows}/\Vert S_{k_{n}}\Vert$
являются
сжимающими в
силу (\ref{e5}).
Отсюда
последовательность
$\left\{  \mathrm{L}_{h}\mathrm{R}_{g}T_{n}^{\leftrightarrows}\right\}  $
ограничена
для
фиксированных
$h,g\in A$. В силу
этого
\[
\mathrm{L}_{Th}\mathrm{R}_{Tg}=\lim_{n\rightarrow\infty}\mathrm{L}{_{T_{n}%
h}\mathrm{R}_{T_{n}g}}=\lim_{n\rightarrow\infty}T_{n}\mathrm{L}_{h}%
\mathrm{R}_{g}T_{n}^{\leftrightarrows}=\lim_{n\rightarrow\infty}%
T\mathrm{L}_{h}\mathrm{R}_{g}T_{n}^{\leftrightarrows}%
\]
и оператор
$\mathrm{L}_{Th}\mathrm{R}_{Tg}$, как
предел
компактных
операторов,
компактен.
Следовательно,
$T\left(  A\right)  $ состоит
из совместно
компактных
элементов
алгебры $A$.

Предположим,
что $r(N)<1$, и
докажем, что
при любых $u,v,x\in
\mathcal{S}(M)$ элемент $u(Tx)v$
квазинильпотентен.
По
построению,
$T=\lim_{n\rightarrow\infty}{\lambda}_{k_{n}}S_{k_{n}}$,
где
\[
{\lambda}_{k_{n}}=\Vert S_{k_{n}}\Vert^{-1}\rightarrow0\text{
при }n\rightarrow\infty\text{ и }S_{k_{n}}%
=\mathrm{L}_{a_{n}}\mathrm{R}_{b_{n}}%
\]
для
некоторых $a_{n}%
,b_{n}\in M^{m\left(  n\right)  }$ и $m\left(  n\right)  \in\mathbb{N}$.
Тогда $ua_{n}xb_{n}v\in\mathcal{S}(M)$
для любого
$n\in\mathbb{N}$ и
последовательность
элементов ${\lambda
}_{k_{n}}ua_{n}xb_{n}v$
стремится к
компактному
элементу $u(Tx)v$
алгебры $A$. Из
\cite[теорема 4.4]{A68}
следует, что
спектр
элемента $u(Tx)v$
не более чем
счётен,
поэтому
спектральный
радиус
непрерывен в
$u(Tx)v$ согласно \cite{New}.
Поскольку $r(N)<1$,
множество $\{\rho
(S):S\in\mathcal{S}(N)\}$
ограничено и
тогда,
учитывая,
что $\mathrm{W}_{ua_{n}xb_{n}v}\in\mathcal{S}\left(
N\right)  $, получим
\[
\rho(u(Tx)v)=\lim_{n\rightarrow\infty}{\lambda}_{k_{n}}\rho(ua_{n}%
xb_{n}v)=\lim_{n\rightarrow\infty}{\lambda}_{k_{n}}\rho(\mathrm{W}%
_{ua_{n}xb_{n}v})^{1/2}=0.
\]
Таким
образом,
множество $\mathcal{S}%
_{+}(M)(Tx)\mathcal{S}_{+}(M)$ при
любом $x\in\mathcal{S}(M)$
состоит из
совместно
компактных
квазинильпотентных
элементов
алгебры $A$. Это
же верно для
его
замыкания $\overline
{\mathcal{S}_{+}(M)(Tx)\mathcal{S}_{+}(M)}$. Так
как
\[
Tx\in\overline{\mathcal{S}_{+}(N)}\mathcal{S}(M)=\overline{\mathrm{L}%
_{\mathcal{S}_{+}(M)}\mathrm{R}_{\mathcal{S}_{+}(M)}}\mathcal{S}%
(M)\subset\overline{\mathcal{S}_{+}(M)\mathcal{S}(M)\mathcal{S}_{+}(M)}%
\subset\overline{\mathcal{S}_{+}(M)},
\]
то множество
$\overline{\mathcal{S}_{+}(M)(Tx)\mathcal{S}_{+}(M)}$
является
полугруппой.
По лемме \ref{semig},
его
замкнутая
линейная
оболочка $J$
состоит из
квазинильпотентных
элементов.
Так как $A=\operatorname*{span}%
\mathcal{S}_{+}(M)$, то $J$ ---
идеал $\overline{A(Tx)A}$
алгебры $A$ и,
таким
образом, $A(Tx)A\subset
\operatorname*{Rad}(A)$. В силу
квазирегулярной
характеризации
радикала
Джекобсона,
$A(Tx)\subset\operatorname*{Rad}(A)$ и, по
той же
причине, $Tx\in\operatorname*{Rad}%
(A)$. Учитывая,
что $A=\operatorname*{span}\mathcal{S}(M)$,
получим, что
$T(A)\subset\operatorname*{Rad}(A)$.
\end{proof}

Бикомпактные
идеалы,
состоящие из
квазинильпотентных
элементов,
мы будем
называть $qb$%
\textit{-идеалами}.

\begin{corollary}
\label{imcor} Пусть $M\in\mathfrak{K}\left(
A\right)  $. Если $~\max\left\{  \rho^{\chi}\left(
M\right)  ,r\left(  M\right)  \right\}  <\rho(M)=1$ и
полугруппа
$\mathcal{S}(\mathrm{L}_{M}\mathrm{R}_{M})$
неограничена,
то ${\mathcal{A}(M)}$ имеет
ненулевой $qb$-идеал.
\end{corollary}

\begin{proof}
Пусть $T$ ---
компактный
оператор,
построенный
в лемме \ref{L4}.
Тогда любой
ненулевой
элемент $Tx\in{\mathcal{A}(M)}$
порождает $qb$-идеал.
\end{proof}

\begin{lemma}
\label{L5} Пусть $M\in\mathfrak{K}\left(
A\right)  $. Если $\mathcal{A}{(M)}$ не
имеет
ненулевых $qb$%
-идеалов, то
равенство
\emph{(\ref{e2})} справедливо.
\end{lemma}

\begin{proof}
Предположим,
что (\ref{e2}) не
выполнено.
Мы можем
считать, что
$\rho(M)=1$. Пусть $N=\mathrm{L}_{M}%
\mathrm{R}_{M}$, тогда $\rho(N)=1$ в
силу леммы \ref{pass}.

Если
полугруппа
$\mathcal{S}(N)$
ограничена,
то $\max\left\{  \rho_{\chi}(N),r\left(  N\right)  \right\}  =1$
согласно
лемме \ref{special}.
Если $\rho_{\chi}(N)=1$, то $\rho^{\chi
}(M)=1$. Если же $\rho_{\chi}(N)<1$,
то $r(N)=1$ и $r(M)=1$ по
лемме \ref{pass}. В
обоих
случаях
выполняется
(\ref{e2}), что
противоречит
сделанному
предположению.
Таким
образом, $\mathcal{S}(N)$
неограничена.
По следствию
\ref{imcor}, $\mathcal{A}\left(  M\right)  $
имеет
ненулевой $qb$%
-идеал, что
невозможно
по условию леммы.
\end{proof}

\begin{theorem}
\label{rho"} Формула $\rho
(M)=\max\{\rho^{\chi}(M),r(M)\}$ верна
для любой
нормированной
алгебры $A$ и
любого $M\in\mathfrak{K}(A)$.
\end{theorem}

\begin{proof}
Как мы уже
отмечали,
можно
считать, что
$A=\mathcal{A}(M)$. Пусть $J=\operatorname*{Rad}%
(A)\cap\mathcal{R}_{\mathrm{hc}}(A)$. Так
как, в силу
теоремы \ref{order}, $J\subset
\mathcal{R}_{\mathrm{cq}}(A)$, то $\rho(M)=\rho(M/J)$
по теореме \ref{cq-rad}.
Далее,
алгебра $A/J$ не
имеет
ненулевых $qb$%
-идеалов.
Действительно,
пусть $I$ ---
какой-нибудь
$qb$-идеал в $A/J$,
тогда его
прообраз $U$ в $A$
содержится в
$\operatorname*{Rad}(A)$ и $\mathcal{R}_{\mathrm{hc}}(A)$,
поскольку
радикалы
устойчивы
относительно
расширений.
Поэтому $U\subset J$,
откуда $I=0$.

Учитывая,
что $A/J=\mathcal{A}\left(  M/J\right)  $, и
применяя
лемму \ref{L5},
заключаем,
что
\[
\rho(M)=\rho(M/J)=\max\{\rho^{\chi}(M/J),r(M/J)\}\leq\max\{\rho^{\chi
}(M),r(M)\}
\]
в силу (\ref{MJM}).
Обратное
неравенство очевидно.
\end{proof}

\subsection{
Операторная
БВ-формула.}

Теперь мы
можем
доказать
формулу (\ref{gbwf}).

\begin{theorem}
\label{top} Если
множество $M\subset
\mathcal{B}(\mathcal{X})$
предкомпактно,
то
\[
\rho(M)=\max\{\rho_{e}(M),r(M)\}.
\]

\end{theorem}

\begin{proof}
По лемме \ref{ineq}, $\Vert
\mathrm{L}_{M}\mathrm{R}_{M}\Vert_{\chi}\leq16\Vert M\Vert_{\chi}\Vert
M\Vert.$ Заменяя $M$ на
$M^{n}$, извлекая
корни $n$-й
степени и
переходя к
пределу при
$n\rightarrow\infty$, придём к
неравенству
$\rho^{\chi}\left(  M\right)  ^{2}=\rho_{\chi}(\mathrm{L}_{M}\mathrm{R}%
_{M})\leq\rho_{\chi}(M)\rho(M).$
Применяя
теорему \ref{rho"},
получим $\rho(M)^{2}\leq\max
\{\rho_{\chi}(M)\rho(M),r(M)^{2}\},$
откуда, в
силу $r(M)\leq\rho(M)$,
\[
\rho(M)\leq\max\{\rho_{\chi}(M),r(M)\}\leq\max\{\rho_{e }(M),r(M)\}
\]
согласно (\ref{he}).
Обратное
неравенство очевидно.
\end{proof}

\subsection{Алгебраическая
БВ-формула.\label{ss43}}

Наша
следующая
цель ---
доказать,
что
\begin{equation}
\rho(M)=\max\{\rho(M/\mathcal{R}_{\mathrm{hc}}^{r}(A)),r(M)\} \label{alg}%
\end{equation}
для любой
нормированной
алгебры $A$ и
для любого
$M\in\mathfrak{K}(A)$. Нам
будет
удобнее
доказывать
(\ref{alg}) в более
гибкой
форме:
\begin{equation}
\rho(M)=\max\{\rho(M/J),r(M)\} \label{id}%
\end{equation}
для любого
замкнутого
гипокомпактного
идеала $J$
алгебры $A$.
Снова
достаточно
показать
справедливость
этой формулы
для
банаховых
алгебр,
поскольку,
как было
отмечено
ранее,
значения
радиусов в (\ref{id})
не меняются
при
пополнении
алгебр и
замыкании
идеала в
пополнении,
в то же время
как
замыкание
идеала гипокомпактно.

Начнём со
случая,
когда идеал $J$
алгебры $A$ бикомпактен.

\begin{lemma}
\label{L6} Пусть $J$ ---
замкнутый
бикомпактный
идеал
алгебры $A$ и $M\in
\mathfrak{K}\left(  A\right)  $. Тогда
\begin{equation}
{\rho}_{e}(\mathrm{L}_{M}\mathrm{R}_{M})\leq\rho(M/J)\rho(M). \label{roe}%
\end{equation}

\end{lemma}

\begin{proof}
Пусть $I=\operatorname{span}J^{2}$.
Докажем
сначала
неравенство
\begin{equation}
\Vert\mathrm{L}_{M}\mathrm{R}_{M}\Vert_{e}\leq3\Vert M/I\Vert\Vert M\Vert.
\label{norm}%
\end{equation}
Пусть $a,b\in M$, $\varepsilon>0$ и
$u,v\in I$, такие что
$\max\left\{  \Vert a-u\Vert,\Vert b-v\Vert\right\}  <\Vert M/I\Vert
+\varepsilon.$ В
частности, $\left\Vert
u\right\Vert <\left\Vert a\right\Vert +\left\Vert a-u\right\Vert
\leq\left\Vert M\right\Vert +\Vert M/I\Vert+\varepsilon\leq2\left\Vert
M\right\Vert +\varepsilon$. Так как $J$
---
бикомпактный
идеал в $A$, то $I$
состоит из
совместно
компактных
элементов
алгебры $A$.
Тогда
оператор $\mathrm{L}%
_{u}\mathrm{R}_{v}$
компактен в $A$
и
\begin{align*}
\left\Vert \mathrm{L}_{a}\mathrm{R}_{b}\right\Vert _{e}  &  \leq\left\Vert
\mathrm{L}_{a}\mathrm{R}_{b}-\mathrm{L}_{u}\mathrm{R}_{v}\right\Vert
=\left\Vert \mathrm{L}_{a-u}\mathrm{R}_{b}+\mathrm{L}_{u}\mathrm{R}%
_{b-v}\right\Vert \leq\left\Vert a-u\right\Vert \left\Vert b\right\Vert
+\left\Vert u\right\Vert \left\Vert b-v\right\Vert \\
&  \leq(\Vert M/I\Vert+\varepsilon)\Vert M\Vert+(2\Vert M\Vert+\varepsilon
)(\Vert M/I\Vert+\varepsilon)\\
&  \leq(\Vert M/I\Vert+\varepsilon)(3\Vert M\Vert+\varepsilon).
\end{align*}
Устремляя $\varepsilon$
к нулю и беря
супремум по
всем $a,b\in M$,
получим (\ref{norm}).

Заменяя $M$ в (\ref{norm})
на $M^{n}$,
извлекая
корни $n$-й
степени и
переходя к
пределу при
$n\rightarrow\infty$, получим
${\rho}_{e}(\mathrm{L}_{M}\mathrm{R}_{M})\leq\rho(M/I)\rho(M)$.
Теперь,
чтобы
получить (\ref{roe}),
заметим, что $J/I$
---
нильпотентный
идеал в $A/I$,
откуда $J/I\in\mathcal{R}_{\mathrm{cq}%
}\left(  A/I\right)  $ и $\rho(M/J)=\rho(M/I)$ в
силу теоремы
\ref{cq-rad}.
\end{proof}

\begin{corollary}
\label{bi} Равенство
\emph{(\ref{id})}
выполнено
для любого
замкнутого
бикомпактного
идеала $J$ и
любого $M\in\mathfrak{K}\left(  A\right)
$.
\end{corollary}

\begin{proof}
Так как $\rho^{\chi}(M)^{2}={\rho}_{\chi
}(\mathrm{L}_{M}\mathrm{R}_{M})\leq{\rho}_{e}(\mathrm{L}_{M}\mathrm{R}%
_{M})\leq\rho(M/J)\rho(M)$
согласно (\ref{roe}),
то, применяя
теорему \ref{rho"} и
неравенство
$r(M)\leq\rho(M)$, получим
\[
\rho(M)^{2}=\max\left\{  \rho^{\chi}(M)^{2},r\left(  M\right)  ^{2}\right\}
\leq\max\{\rho(M/J)\rho(M),r(M)\rho(M)\},
\]
что
немедленно
влечет
равенство (\ref{id}).
\end{proof}

\begin{lemma}
\label{step} Пусть $I\subset K$ ---
замкнутые
идеалы
алгебры $A$ и $M\in
\mathfrak{K}\left(  A\right)  $. Если
алгебра $K/I$
бикомпактна
и равенство
\emph{(\ref{id})}
выполнено
при $J=I$, то оно
выполнено и
при $J=K$.
\end{lemma}

\begin{proof}
Изоморфизм
$A/K\rightarrow(A/I)/(K/I)$
обеспечивает
равенство $\rho
(M/K)=\rho((M/I)/(K/I)),$ что, в
свою
очередь,
даёт оценку
\begin{align*}
\rho(M)  &  =\max\{\rho(M/I),r(M)\}=\max\{\max\{\rho
((M/I)/(K/I)),r(M/I)\},r(M)\}\\
&  \leq\max\{\rho(M/K),r(M)\}.
\end{align*}
Обратное
неравенство очевидно.
\end{proof}

\begin{lemma}
\label{cont} Если $J=\overline{\bigcup J_{\alpha}}$,
где $(J_{\alpha})$ ---
направленная
по
возрастанию
сеть
замкнутых
идеалов
алгебры $A$, то
для любого
$M\in\mathfrak{K}(A)$%
\begin{align}
\Vert M/J\Vert &  =\lim_{\alpha}\Vert M/J_{\alpha}\Vert=\inf_{\alpha}\Vert
M/J_{\alpha}\Vert,\label{econt1}\\
\rho(M/J)  &  =\lim_{\alpha}\rho(M/J_{\alpha})=\inf_{\alpha}\rho(M/J_{\alpha
}). \label{econt}%
\end{align}

\end{lemma}

\begin{proof}
Для любого $\alpha$
выполняется
неравенство
$\left\Vert M/J\right\Vert \leq\left\Vert M/J_{\alpha}\right\Vert ,$
откуда $\rho(M/J)=\inf_{n}\left\Vert
M^{n}/J\right\Vert ^{1/n}\leq\inf_{\alpha}\inf_{n}\left\Vert M^{n}/J_{\alpha
}\right\Vert ^{1/n}=\inf_{\alpha}\rho(M/J_{\alpha})$ и,
кроме того, $\left\Vert
M/J\right\Vert \leq\inf_{\alpha}\left\Vert M/J_{\alpha}\right\Vert
\leq\underset{\alpha}{\,\lim\inf}\left\Vert M/J_{\alpha}\right\Vert .$ С
другой
стороны,
легко
проверить,
что для
любых $a\in M$ и $\varepsilon>0$
найдется
такое $\alpha=\alpha(a,\varepsilon)$,
что
\begin{equation}
\left\Vert a/J_{\alpha}\right\Vert \leq\left\Vert a/J\right\Vert
+\varepsilon\leq\left\Vert M/J\right\Vert +\varepsilon. \label{econt2}%
\end{equation}
Выберем в $M$
конечную $\varepsilon
$-сеть $N$; так
как для
любого $c\in A$ из
$\beta>\alpha$ следует
неравенство
$\left\Vert c/J_{\beta}\right\Vert \leq\left\Vert c/J_{\alpha}\right\Vert ,$
то, выбирая $\gamma
\geq\max\left\{  \alpha(a,\varepsilon):a\in N\right\}  $,
получим, в
силу (\ref{econt2}), что
$\operatorname*{dist}(b/J_{\gamma},N/J_{\gamma})\leq\operatorname*{dist}%
(b,N)\leq\varepsilon$ для
каждого $b\in M$, и
потому $\left\Vert M/J_{\gamma}\right\Vert
\leq\left\Vert N/J_{\gamma}\right\Vert +\varepsilon\leq\left\Vert
M/J\right\Vert +2\varepsilon.$
Следовательно,
\begin{equation}
\inf_{\alpha}\left\Vert M/J_{\alpha}\right\Vert \leq\underset{\alpha}%
{\,\lim\sup\,}\left\Vert M/J_{\alpha}\right\Vert \leq\left\Vert M/J\right\Vert
, \label{econt3}%
\end{equation}
откуда
вытекает (\ref{econt1}).
Выберем
такое $n\in\mathbb{N}$, что
$\left\Vert M^{n}/J\right\Vert ^{1/n}\leq\rho(M/J)+\varepsilon.$
Тогда, в силу
неравенства
(\ref{econt3}),
применённого
к $M^{n}$, имеем
\begin{align*}
\inf_{\alpha}\rho(M/J_{\alpha})  &  \leq\lim\sup\,\,\rho(M/J_{\alpha}%
)\leq\underset{\alpha}{\,\lim\sup\,}\Vert M^{n}/J_{\alpha}\Vert^{1/n}%
\leq\left\Vert M^{n}/J\right\Vert ^{1/n}\\
&  \leq\rho(M/J)+\varepsilon.
\end{align*}
Это
доказывает
справедливость
равенства (\ref{econt}).
\end{proof}

Теперь мы
докажем
формулу (\ref{id}) и,
тем самым,
формулу (\ref{alg}).

\begin{theorem}
\label{algGBWF}Формула $\rho
(M)=\max\{\rho(M/J),r(M)\}$\textrm{ }%
справедлива
для любого
замкнутого
гипокомпактного
идеала $J$ и
любого $M\in\mathfrak{K}\left(  A\right)
$.
\end{theorem}

\begin{proof}
Найдётся
возрастающая
трансфинитная
цепочка $\{J_{\alpha}\}_{\alpha
\leq\beta}$ замкнутых
идеалов в $A$,
такая что $J_{0}=0$,
$J_{\beta}=J$, и все $J_{\alpha+1}/J_{\alpha}$
бикомпактны.
Если есть
ординалы $\alpha$,
для которых
(\ref{id}) при $J=J_{\alpha}$
ложно, то
пусть $\gamma$ ---
наименьший
из них. Он не
является
предельным в
силу леммы \ref{cont},
и не имеет
предыдущего
в силу леммы
\ref{step}. Значит, (\ref{id})
выполнено
для всех $\alpha$.
\end{proof}

Из теоремы
Вала \cite{Vala}
следует, что
$\mathcal{K}(\mathcal{X})$ ---
бикомпактный
идеал в $\mathcal{B}(\mathcal{X})$.
Пусть $\mathcal{K}_{w}(\mathcal{X})$ ---
идеал всех
слабо
компактных
операторов в
$\mathcal{X}$. Как
известно (см.
\cite[следствие
6.8.13]{DS}), для
широкого
класса
банаховых
пространств
$\mathcal{X}$,
включающего
все
пространства
$L^{1}$
суммируемых
функций,
произведение
двух
операторов
из $\mathcal{K}_{w}(\mathcal{X})$
компактно,
так что
фактор-алгебра
$\mathcal{K}_{w}(\mathcal{X})/\mathcal{K}(\mathcal{X})$
имеет
тривиальное
умножение, а
значит,
бикомпактна.
Поэтому для
таких $\mathcal{X}$
идеал $\mathcal{K}_{w}(\mathcal{X})$
гипокомпактен
и потому
содержится в
$\mathcal{R}_{\mathrm{hc}}(\mathcal{B}(\mathcal{X}))$.
Таким
образом,
идеал $\mathcal{R}_{\mathrm{hc}}(\mathcal{B}%
(\mathcal{X}))$ может
строго
включать $\mathcal{K}%
(\mathcal{X})$.

\section{Дальнейшее
усиление
алгебраической
БВ-формулы}

\subsection{ В поисках
идеала.}

Формулы типа
(\ref{id}) тем
полезнее,
чем больше
входящий в
них идеал $J$.
При $J=A$ такая
формула
означает,
что в
алгебре $A$
величины $\rho(M)$ и
$r(M)$ совпадают
на $\mathfrak{K}(A)$, а при $J=0$
равенство (\ref{id})
тривиально.
Поэтому
естественно
ставить
задачу об
отыскании
(или хотя бы о
существовании)
наибольшего
идеала с
таким
свойством.
Обозначим
через $BW(A)$
совокупность
всех
замкнутых
идеалов $J\subset A$,
для которых
равенство (\ref{id})
верно для
любого $M\in\mathfrak{K}(A)$;
такие идеалы
будем
называть
\textit{BW-идеалами}.
Удобно
выделить
следующий
результат,
который
легко доказывается.

\begin{lemma}
\label{or}Пусть $I\subset J$ ---
замкнутые
идеалы в $A$.
Если $J\in BW(A)$, то $I\in BW(A)$.
\end{lemma}

Следующий
результат
доказывается
аналогично
лемме \ref{step}.

\begin{lemma}
\label{BW-step} Пусть $I\subset J$ ---
замкнутые
идеалы в $A$.
Если $I\in BW(A)$ и $J/I\in BW(A/I)$,
то $J\in BW(A)$.
\end{lemma}

\begin{lemma}
\label{BW-cont} Если $J=\overline{\bigcup J_{\alpha}}%
$, где $(J_{\alpha})$ ---
направленная
по
возрастанию
сеть
замкнутых
идеалов
алгебры $A$, и $J_{\alpha}\in
BW(A)$ для всех $\alpha$,
то $J\in BW(A)$.
\end{lemma}

\begin{proof}
Пусть $M\in\mathfrak{K}(A)$.
Можно
считать, что
$r(M)<\rho(M)$ (в
противном
случае
равенство (\ref{id})
очевидно).
Следовательно,
из $J_{\alpha}\in BW(A)$
вытекает,
что $\rho(M/J_{\alpha})=\rho(M)$ для
всех $\alpha$.
Применяя
лемму \ref{cont},
получим, что
$\rho(M/J)=\lim_{\alpha}\rho(M/J_{\alpha})=\rho(M)$.
\end{proof}

Применяя
леммы \ref{BW-step} и \ref{BW-cont},
получаем по
индукции
следующий результат:

\begin{corollary}
\label{venec} Если
возрастающая
трансфинитная
цепочка
замкнутых
идеалов $(J_{\alpha})_{\alpha
\leq\gamma}$ алгебры $A$
такова, что $J_{0}=0$
и $J_{\alpha+1}/J_{\alpha}\in BW(A/J_{\alpha})$ для
всех $\alpha<\gamma$, то $J_{\gamma}\in
BW(A)$.
\end{corollary}

По теореме
\ref{algGBWF}, все
замкнутые
гипокомпактные
идеалы
являются
BW-идеалами.
Докажем, что
в $BW(A)$ попадают
также все
замкнутые
коммутативные
идеалы. На
деле, для них
равенство (\ref{id})
принимает
более
сильную
форму.
Напомним,
что $r_{1}(M)$
определяется
(см. (\ref{nach2})) как $\sup\{\rho(a):a\in M\}$.

\begin{lemma}
\label{centr} Если $J$ ---
замкнутый
центральный
идеал в $A$ и $M\in\mathfrak{K}%
\left(  A\right)  $, то
\begin{equation}
\rho(M)=\max\{\rho(M/J),r_{1}(\mathrm{L}_{M}|J)\}=\max\{\rho(M/J),r_{1}(M)\}.
\label{central}%
\end{equation}

\end{lemma}

\begin{proof}
Cемейство
операторов
$\mathrm{L}_{M}$
действует на
нормированном
пространстве
$A$ и оставляет
инвариантным
замкнутое
подпространство
$J$. По лемме 4.2 \cite{ShT2000}
(которую
легко
адаптировать
к неполным
пространствам),%
\[
\rho(\mathrm{L}_{M})=\max\{\rho(\mathrm{L}_{M}|J),\rho(\mathrm{L}_{M}|\left(
A/J\right)  )\},
\]
где $\mathrm{L}_{M}|J$ ---
сужение $\mathrm{L}_{M}$
на $J$ и $\mathrm{L}_{M}|\left(  A/J\right)  $ ---
семейство,
индуцированное
в факторе $A/J$.
Ясно, что $\mathrm{L}_{M}|\left(
A/J\right)  =\mathrm{L}_{M/J}$, и тогда
$\rho(\mathrm{L}_{M}|\left(  A/J\right)  ))=\rho(M/J)$. На
подпространстве
$J$ операторы
умножения
коммутируют:
$\mathrm{L}_{a}\mathrm{L}_{b}x=abx=bxa=bax=\mathrm{L}_{b}\mathrm{L}_{a}x$.
Тогда $\rho(\mathrm{L}_{M}|J)=r_{1}%
(\mathrm{L}_{M}|J)$ по \cite[лемма
2.8]{ShT2002}. Остаётся
учесть, что
$r_{1}(\mathrm{L}_{M}|J)\leq r_{1}(\mathrm{L}_{M})=r_{1}(M)$ и
$\rho(M)=\rho(\mathrm{L}_{M})$.
\end{proof}

Алгебра $A$
называется
\textit{полупримарной}%
, если она не
имеет
нильпотентных
идеалов.
Эквивалентное
условие: из
равенства $aAa=0$
следует, что $a=0$.

\begin{lemma}
\label{com-id} Если
алгебра $A$
полупримарна,
то любой её
коммутативный
идеал $J$
централен.
Как
следствие,
среди всех
её
коммутативных
идеалов есть наибольший.
\end{lemma}

\begin{proof}
Пусть $x,y\in J$, $a\in A$,
тогда $x[a,y]=xay-xya=yxa-yxa=0$.
Взяв $x=[a,y]b$, где $b\in A$,
получим $[a,y]A[a,y]=0$,
откуда $[a,y]=0$. Это
означает,
что $J$ централен.
\end{proof}

Будем
говорить,
что
подмножество
$K\subset A$
\textit{коммутативно
по модулю}
идеала $J$, если
$ab-ba\in J$ для любых $a,b\in
K$.

\begin{theorem}
\label{commodR}

\emph{(i)} Любая
нормированная
алгебра
имеет
наибольший
идеал $\mathcal{R}_{\mathrm{cq}}^{a}(A)$,
коммутативный
по модулю $\mathcal{R}%
_{\mathrm{cq}}(A)$.

\emph{(ii)} Если
замкнутый
идеал $J$
нормированной
алгебры $A$
коммутативен
по модулю $\mathcal{R}%
_{\mathrm{cq}}(A)$, то
равенство
\emph{(\ref{central})}
выполнено
для любого
$M\in\mathfrak{K}(A)$.
\end{theorem}

\begin{proof}
Так как
любой
нильпотентный
идеал
компактно
квазинильпотентен,
то алгебра
$B=A/\mathcal{R}_{\mathrm{cq}}(A)$
полупримарна.
По лемме \ref{com-id},
она имеет
наибольший
коммутативный
идеал $I(B)$; его
прообраз $\mathcal{R}%
_{\mathrm{cq}}^{a}(A)$ в $A$ ---
наибольший
идеал,
коммутативный
по модулю $R_{\mathrm{cq}}(A)$.

Утверждение
(ii) достаточно
доказать для
$J=\mathcal{R}_{\mathrm{cq}}^{a}(A)$.
Используя
центральность
идеала $I(B)$ и
лемму \ref{centr},
имеем
\begin{align*}
\rho(M)  &  =\rho(M/\mathcal{R}_{\mathrm{cq}}(A))=\max\{\rho_{I(B)}%
(M/\mathcal{R}_{\mathrm{cq}}(A)),r_{1}(M/\mathcal{R}_{\mathrm{cq}}(A))\}\\
&  =\max\{\rho(M/J),r_{1}(M)\},
\end{align*}
где $\rho_{I(B)}\left(  N\right)  $
обозначает
$\rho\left(  N/I(B)\right)  $ для
любого $N\in\mathfrak{K}\left(  B\right)
$.
\end{proof}

Для
банаховых
алгебр
справедливо
и обратное утверждение.

\begin{theorem}
\label{com-mod} Для того,
чтобы в
банаховой
алгебре $A$
выполнялось
условие:
\begin{equation}
\rho(M)=\sup\{\rho(a):a\in M\}\text{ для
всех }M\in\mathfrak{K}(A), \label{r-max}%
\end{equation}
необходимо и
достаточно,
чтобы она
была
коммутативной
по модулю $\mathcal{R}%
_{\mathrm{cq}}(A)$.
\end{theorem}

\begin{proof}
Достаточность
следует из
теоремы \ref{commodR}.

Пусть
выполнено (\ref{r-max}).
Для любых $a,b\in A$,
положим $(\mathrm{ad}~b)a=[a,b]$,
тогда при
любом $\lambda\neq0$
отображение
$\exp(\lambda(\mathrm{ad}~b))$
является
автоморфизмом
алгебры $A$, и
потому $\rho(\exp(\lambda(\mathrm{ad}%
~b))a)=\rho(a)$. Пусть $M\in\mathfrak{K}(A)$.
Из формулы
(\ref{r-max}),
применённой
к множеству
$N=M\cup\{\lambda^{-1}a,\lambda^{-1}\exp(\lambda(\mathrm{ad}~b))a\}$,
получаем,
что
\[
\rho(N)=\max\{\rho(a)/|\lambda|,r_{1}(M)\}.
\]
Обозначим
правую часть
равенства
через $C(\lambda)$.
Так как
 $\frac{1}{2}(\exp
(\lambda(\mathrm{ad}~b))a-a)/\lambda$
принадлежит
абсолютно
выпуклой
оболочке $\mathrm{abs}(N)$
множества $N$ и
$\rho(\mathrm{abs}(N))=\rho(N)$ (см.
\cite[предложение
2.6]{ShT2000}), получим,
что
\[
\rho\left(  \left\{  (\left(  \exp(\lambda(\mathrm{ad}~b))a-a)/\lambda\right)
/2\right\}  \cup M\right)  \leq C(\lambda).
\]
При $\lambda\rightarrow\infty$
верхний
предел
функции $f(\lambda)=\rho\left(
\left\{  \frac{1}{2}(\exp(\lambda(\mathrm{ad}~b))a-a)/\lambda\right\}  \cup
M\right)  $ не
превосходит
числа $r_{1}(M)$.
Будучи
субгармонической
(см. \cite[теорема
3.5]{ShT2000}), эта
функция
постоянна и
везде не
превосходит
$r_{1}(M)$. В
частности,
\[
f(0)=\rho\left(  \left\{  \lbrack a,b]/2\right\}  \cup M\right)  \leq
r_{1}(M)\leq\rho(M).
\]
Так как $M$
произвольно,
это
означает,
что $[a,b]\in\mathcal{R}_{\mathrm{cq}}(A)$ (см.
теорему \ref{cq-rad}(b)).
Следовательно,
$A$
коммутативна
по модулю
идеала $\mathcal{R}_{\mathrm{cq}}(A)$.
\end{proof}

Таким
образом, $BW(A)$
содержит все
замкнутые
идеалы,
которые
получаются
из
бикомпактных,
коммутативных
и компактно
квазинильпотентных
с помощью
последовательного
применения
операции расширения.

По лемме \ref{BW-cont},
всякая
цепочка $BW$%
-идеалов
имеет
мажоранту,
являющуюся
$BW$-идеалом.
Следовательно,
по лемме
Цорна, в
любой
нормированной
алгебре $A$
есть
максимальные
$BW$-идеалы, и
фактор-алгебра
по любому из
этих идеалов
не имеет $BW$%
-идеалов по
лемме \ref{BW-step}.
Неизвестно,
однако,
существует
ли
наибольший
$BW$-идеал.
Эквивалентный
вопрос:
является ли
замыкание
суммы $BW$%
-идеалов $BW$%
-идеалом? Мы
увидим, что
ответы на
эти вопросы
станут
утвердительными,
если вместо
идеалов
рассматривать
радикалы.

\subsection{ От идеалов
к радикалам.}

Скажем, что
ТР $P$ является
$BW$\textit{-ра%
\-%
ди%
\-%
ка%
\-%
лом}, если все
замкнутые $P$%
-радикальные
идеалы
являются $BW$-идеалами.

\begin{theorem}
\label{BW-radical}
Cуществует
наибольший
$BW$-радикал $\mathcal{R}%
_{\mathrm{bw}}$; всякий
топологический
радикал, не
превосходяший
$\mathcal{R}_{\mathrm{bw}}$,
является $BW$-радикалом.
\end{theorem}

\begin{proof}
Обозначим
через $\mathcal{R}_{\mathrm{bw}}$
точную
верхнюю
грань
семейства
всех $BW$%
-радикалов
(см. лемму \ref{sup});
требуется
доказать,
что $\mathcal{R}_{\mathrm{bw}}(A)$
является $BW$%
-идеалом для
любой
нормированной
алгебры $A$.

По лемме \ref{sup},
существует
возрастающая
трансфинитная
цепочка
замкнутых
идеалов $(J_{\alpha})_{\alpha
\leq\gamma}$, такая что
$J_{1}=0$, $J_{\gamma}=\mathcal{R}_{\mathrm{bw}}(A)$ и
каждый
фактор $J_{\alpha+1}/J_{\alpha}$
является $P$%
-радикальным
для
некоторого
$BW$-радикала $P$.
Значит, $J_{\alpha+1}/J_{\alpha}$
является $BW$%
-идеалом в $A/J_{\alpha}$.
По следствию
\ref{venec}, $\mathcal{R}_{\mathrm{bw}}(A)$ --- $BW$-идеал.

Если $P$ --- ТР и $P\leq
\mathcal{R}_{\mathrm{bw}}$, то $P$ --- $BW$%
-радикал в
силу леммы \ref{or}.
\end{proof}

Согласно
теореме, $\rho(M)=\max
\{\rho(M/\mathcal{R}_{\mathrm{bw}}(A)),r(M)\}$ для
любого $M\in\mathfrak{K}\left(  A\right)
$ и любой
алгебры $A$. В
частности,
справедлив
следующий результат.

\begin{corollary}
$\mathcal{R}_{\mathrm{bw}}(A)\cap\operatorname*{Rad}^{r}\left(  A\right)
\subset\mathcal{R}_{\mathrm{cq}}(A)$ для
любой
алгебры $A$.
\end{corollary}

В силу (\ref{alg}),
регулярный
гипокомпактный
радикал $\mathcal{R}%
_{\mathrm{hc}}^{r}$ --- $BW$%
-радикал.
Покажем, что
отображение
$A\longmapsto\mathcal{R}_{\mathrm{cq}}^{a}(A)$ ---
также $BW$-радикал.

\begin{lemma}
\label{indiv} $\mathcal{R}_{\mathrm{cq}}^{a}(A)=\{x\in A:x[a,b]\in
\mathcal{R}_{\mathrm{cq}}(A)\text{ для
всех }a,b\in A\}.$
\end{lemma}

\begin{proof}
Пусть $P(A)$ ---
правая часть
доказываемого
равенства.
Легко
видеть, что $P(A)$ ---
идеал в $A$. Если
$x\in P(A)$, $a,b\in A$, то $[x,a]b[x,a]\in\mathcal{R}%
_{\mathrm{cq}}(A)$
(поскольку $[x,a]b\in
P(A)$), то есть,
образ
элемента $[x,a]$ в
$A/\mathcal{R}_{\mathrm{cq}}(A)$
порождает
нильпотентный
идеал. Так
как $A/\mathcal{R}_{\mathrm{cq}}(A)$
полупримарна,
то этот
образ равен
нулю и тогда $x$
централен по
модулю $\mathcal{R}_{\mathrm{cq}}(A)$.
Следовательно,
$P(A)\subset\mathcal{R}_{\mathrm{cq}}^{a}(A)$.

Обратно,
пусть $x\in\mathcal{R}_{\mathrm{cq}}^{a}(A)$.
Так как $x$
централeн по
модулю $\mathcal{R}_{\mathrm{cq}}(A)$,
то $x[a,b]=[xa,b]-[x,b]a\in\mathcal{R}_{\mathrm{cq}}(A)$
для любых $a,b\in A$,
откуда $\mathcal{R}_{\mathrm{cq}}%
^{a}(A)\subset P(A)$.
\end{proof}

Из этой
леммы и
теоремы \ref{cq-rad}(d)
сразу
вытекает утверждение:

\begin{corollary}
\label{acq} $\mathcal{R}_{\mathrm{cq}}^{a}(B )=B\cap\mathcal{R}%
_{\mathrm{cq}}^{a}(A)$ для
любой
плотной
подалгебры
$B $
нормированной
алгебры $A$.
\end{corollary}

\begin{theorem}
\label{rcrad}
Отображение
$A\longmapsto\mathcal{R}_{\mathrm{cq}}^{a}(A)$ ---
наследственный
топологический
радикал.
\end{theorem}

\begin{proof}
Мы должны
проверить
условия (R1), (R2) и (R5)
для $\mathcal{R}_{\mathrm{cq}}^{a}$,
причём
достаточно
это сделать
для
банаховых
алгебр. В
самом деле,
если
показать,
что $\mathcal{R}_{\mathrm{cq}}^{a}$ ---
НТР в классе
банаховых
алгебр, то
регулярное
расширение
$\mathcal{R}_{\mathrm{cq}}^{ar}$ --- НТР в
силу
\cite[теорема 2.21]{rad1}.
Но $\mathcal{R}_{\mathrm{cq}}^{a}=\mathcal{R}_{\mathrm{cq}}^{ar}$
по следствию
\ref{acq}.

Если $f:A\rightarrow B$ ---
непрерывный
эпиморфизм
банаховых
алгебр, то,
полагая $J=f(R_{cq}^{a}(A))$,
получим $[J,J]\subset f([\mathcal{R}%
_{\mathrm{cq}}^{a}(A),\mathcal{R}_{\mathrm{cq}}^{a}(A)])\subset f(\mathcal{R}%
_{\mathrm{cq}}(A))\subset\mathcal{R}_{\mathrm{cq}}(B)$, то
есть, идеал $J$
коммутативен
по модулю $\mathcal{R}%
_{\mathrm{cq}}(B)$. Значит, $f(\mathcal{R}%
_{\mathrm{cq}}^{a}(A))\subset\mathcal{R}_{\mathrm{cq}}^{a}(B)$, то
есть,
выполнено
условие (R1).

Пусть $I=\mathcal{R}_{\mathrm{cq}}^{a}(A)$,
$B=A/I$, $K=\{x\in A:q_{I}(x)\in\mathcal{R}_{\mathrm{cq}}^{a}(B)\}$;
требуется
доказать,
что $K\subset I$. Ясно,
что (\ref{central})
справедливо
при $J=I$ и
выполнено в
алгебре $B$ при
$J=K/I$. Рассуждая
как в лемме
\ref{BW-step} (то есть,
как в лемме \ref{step}),
получим, что
оно
выполнено
при $J=K$, откуда
$\rho(M)=r_{1}(M)$ при $M\in\mathfrak{K}\left(  K\right)  $.
По теореме
\ref{com-mod}, идеал $K$
коммутативен
по модулю $\mathcal{R}%
_{\mathrm{cq}}(K)=K\cap\mathcal{R}_{\mathrm{cq}}(A)$, а
значит, и по
модулю $\mathcal{R}_{\mathrm{cq}}(A)$,
то есть, $K\subset\mathcal{R}%
_{\mathrm{cq}}^{a}(A)=I$. Условие
(R2) доказано.

Пусть $I$ ---
замкнутый
идеал
банаховой
алгебры $A$. Из
леммы \ref{indiv}
следует, что
$\mathcal{R}_{\mathrm{cq}}^{a}(I)$
инвариантен
относительно
непрерывных
автоморфизмов
алгебры $I$, а
потому и
относительно
её
ограниченных
дифференцирований
(поскольку
они ---
генераторы
групп
автоморфизмов).
Следовательно,
если $x\in\mathcal{R}_{\mathrm{cq}}^{a}(I)$, $a\in A$,
то $[a,x]\in\mathcal{R}_{\mathrm{cq}}^{a}(I)$. С
другой
стороны, из
леммы \ref{indiv} (и
того, что $\mathcal{R}_{\mathrm{cq}%
}(I)$ --- идеал в $A$)
следует, что
$ax[v,w]\in\mathcal{R}_{\mathrm{cq}}(I)$ для
любых $v,w\in I$,
поэтому $ax\in\mathcal{R}%
_{\mathrm{cq}}^{a}(I)$. Отсюда
$xa\in\mathcal{R}_{\mathrm{cq}}^{a}(I)$. Мы
доказали,
что $\mathcal{R}_{\mathrm{cq}}^{a}(I)$ ---
идеал в $A$.

Так как $\mathcal{R}_{\mathrm{cq}%
}(I)\subset\mathcal{R}_{\mathrm{cq}}(A)$ и
идеал $\mathcal{R}_{\mathrm{cq}}^{a}(I)$
коммутативен
по модулю $\mathcal{R}%
_{\mathrm{cq}}(I)$, то он
коммутативен
и по модулю
$\mathcal{R}_{\mathrm{cq}}(A)$, а
значит,
содержится в
$\mathcal{R}_{\mathrm{cq}}^{a}(A)$. Таким
образом, $\mathcal{R}_{\mathrm{cq}%
}^{a}(I)\subset I\cap\mathcal{R}_{\mathrm{cq}}^{a}(A)$. С
другой
стороны,
если $x\in I\cap\mathcal{R}_{\mathrm{cq}}^{a}(A)$,
то для любых
$a,b\in I$, по лемме \ref{indiv},
$x[a,b]\subset I\cap\mathcal{R}_{\mathrm{cq}}(A)=\mathcal{R}_{\mathrm{cq}}(I)$
и,
следовательно,
$x\in\mathcal{R}_{\mathrm{cq}}^{a}(I)$. Это
доказывает
включение $I\cap
\mathcal{R}_{\mathrm{cq}}^{a}(A)\subset\mathcal{R}_{\mathrm{cq}}^{a}(I)$. (R5) доказано.
\end{proof}

\begin{corollary}
Радикал $\mathcal{R}_{\mathrm{cq}%
}^{a}$ равномерен
и является $BW$-радикалом.
\end{corollary}

\begin{proof}
Для
равномерности
достаточно
показать в
силу
следствия \ref{acq},
что любая
замкнутая
подалгебра
$\mathcal{R}_{\mathrm{cq}}^{a}$-ра%
\-%
ди%
\-%
каль%
\-%
ной
банаховой
алгебры $\mathcal{R}_{\mathrm{cq}%
}^{a}$-радикальна.
Но это
очевидно в
силу теоремы
\ref{com-mod}.

Далее, $\mathcal{R}_{\mathrm{cq}}^{a}$
является $BW$%
-радикалом
согласно
теоремам \ref{commodR} и
\ref{rcrad}.
\end{proof}

Обозначим
радикал $\mathcal{R}_{\mathrm{hc}%
}^{r}\vee\mathcal{R}_{\mathrm{cq}}^{a}$ через
$\mathcal{R}_{\mathrm{hc}^{r}+\mathrm{cq}^{a}}$. Ясно,
что он
отличен от
$\mathcal{R}_{\mathrm{hc}}^{r}$ и $\mathcal{R}_{\mathrm{cq}}^{a}$,
и что $\mathcal{R}_{\mathrm{hc}^{r}+\mathrm{cq}^{a}%
}\leq\mathcal{R}_{\mathrm{bw}}$.
Следующий
результат
вытекает из
теоремы \ref{BW-radical}.

\begin{corollary}
Радикал $\mathcal{R}_{\mathrm{hc}%
^{r}+\mathrm{cq}^{a}}$
является $BW$-радикалом.
\end{corollary}

\begin{corollary}
\label{racl}Замыкание
$\mathcal{R}_{\mathrm{hc}^{r}+\mathrm{cq}^{a}}$%
-радикальной
подалгебры
$\mathcal{R}_{\mathrm{hc}^{r}+\mathrm{cq}^{a}}$-ра\textit{%
\-%
}ди\textit{%
\-%
}каль\textit{%
\-%
}но.
\end{corollary}

\begin{proof}
Заметим, что
доказываемым
свойством
обладают
радикалы $\mathcal{R}%
_{\mathrm{hc}}^{r}$ и $\mathcal{R}_{\mathrm{cq}}^{a}$.
Поэтому из
леммы \ref{sup}
следует, что
замыкание $B$
произвольной
$\mathcal{R}_{\mathrm{hc}^{r}+\mathrm{cq}^{a}}$%
-радикальной
подалгебры
обладает
цепочкой
замкнутых
идеалов $\left(  J_{\alpha}\right)
_{\alpha\leq\gamma}$, таких
что $J_{0}=0$, $J_{\gamma}=B$ и
каждый
фактор $J_{\alpha+1}/J_{\alpha}$
либо $\mathcal{R}_{\mathrm{hc}^{r}}$%
-радикален,
либо $\mathcal{R}_{\mathrm{cq}^{a}}$%
-радикален, а
значит, $\mathcal{R}_{\mathrm{hc}%
^{r}+\mathrm{cq}^{a}}$%
-радикален.
Тогда
алгебра $B$ $\mathcal{R}%
_{\mathrm{hc}^{r}+\mathrm{cq}^{a}}$%
-радикальна
по лемме \ref{transext}.
\end{proof}

\section{Роль
БВ-формул в
вопросе о
непрерывности
ССР}

В
\cite[предложение
3.1]{ShT2000} было
показано,
что ССР ---
полунепрерывная
сверху
функция
ограниченного
подмножества
нормированной
алгебры. Это
означает,
что $\underset{n\rightarrow\infty}{\limsup\,\,}\rho
(M_{n})\leq\rho(M)$, если
последовательность
ограниченных
подмножеств
$M_{n}$ стремится
к $M$ в смысле
расстояния
Хаусдорфа.
Скажем, что
ССР
\textit{непрерывен
в} $M\subset A$, если $\rho(M_{n})\rightarrow
\rho(M)$, когда $M_{n}$
стремится к $M$.

Следующее
утверждение очевидно.

\begin{lemma}
\label{app}Пусть $M \subset A$
ограничено и
$J$ --- замкнутый
идеал в $A $.
Если $\rho(M)=$ $\rho(M/J)$ и
ССР
непрерывен в
$M/J$, то ССР
непрерывен в
$M $.
\end{lemma}

Хорошо
известно,
что оператор
$T$ --- точка
непрерывности
спектрального
радиуса (в
обычном
смысле), если
$\rho_{e}(T)<\rho(T)$, а также
что
спектральный
радиус
непрерывен
на
коммутативных
нормированных
алгебрах. Мы
получим
сейчас
аналогичные
результаты
для ССР.

\begin{lemma}
\label{appl} Пусть $M\in\mathfrak{K}\left(
A\right)  $. ССР
непрерывен в
$M,$ если
выполнено
хотя бы одно
из условий:

\begin{enumerate}
\item[(i)] $A=\mathcal{B}\left(  \mathcal{X}\right)  $ и $\rho_{\chi
}(M)<\rho(M)$;

\item[(ii)]
существует
замкнутый
центральный
идеал $J$
алгебры $A$,
такой что $\rho(M/J)<\rho(M)$.
\end{enumerate}
\end{lemma}

\begin{proof}
Пусть $M_{n}\rightarrow M$.
Покажем, что
$\underset{n\rightarrow\infty}{\liminf\,}\rho(M_{n})\geq\rho(M)$.

(i) Допустим
противное.
Умножая на
константу и
заменяя
последовательность
$\left\{  M_{n}\right\}  $ её
подпоследовательностью,
мы можем
считать, что
$\rho_{\chi}(M)<1$ и $\lim\rho(M_{n})<1<\rho(M).$ По
теореме \ref{top}, $\rho(M)=r(M).$
Следовательно,
$\sup\{\rho(T):T\in M^{k}\}>1$ при
некотором $k$, и
найдется $T\in M^{k}$,
такой что $\rho(T)>1.$
Так как $\rho_{\chi}(M)<1$, то
$\Vert M^{n}\Vert_{\chi}<1$ при
достаточно
больших $n$ и
тогда $\Vert T^{n}\Vert_{\chi}\leq\Vert
M^{nk}\Vert_{\chi}<1$, откуда $\rho_{\chi
}(T)\leq1.$ Поскольку
$\rho_{\chi}\left(  T\right)  =\rho_{e}\left(  T\right)  $ (см.
\cite[следствие 6.4]%
{LS}), получим,
что $\rho_{e}(T)<\rho(T).$ По
условию,
найдутся
такие $T_{n}\in M_{n}^{k}$,
что $T_{n}\rightarrow T$ и,
следовательно,
$\rho(T_{n})\rightarrow\rho(T)$ при $n\rightarrow\infty$.
Но это
невозможно,
поскольку $\rho
(T_{n})\leq\rho(M_{n}^{k})=\rho(M_{n})^{k}$ и $\lim_{n}\rho(M_{n}%
)^{k}<1$.

(ii) Согласно (\ref{central}),
$\rho(M)=r_{1}\left(  \mathrm{L}_{M}|J\right)  $. Так
как
операторы
левого
умножения,
суженные на $J$,
образуют
коммутативную
алгебру и $\mathrm{L}_{M_{n}%
}|J$ стремится к
$\mathrm{L}_{M}|J$, то из
\cite[лемма 2.7]{ShT2002}
вытекает,
что $r_{1}\left(  \mathrm{L}_{M_{n}}|J\right)  \rightarrow
r_{1}\left(  \mathrm{L}_{M}|J\right)  $ при $n\rightarrow
\infty$. Поскольку
$\rho(M_{n})\geq r_{1}\left(  \mathrm{L}_{M_{n}}|J\right)  $
для любого $n$,
получим $\underset{n\rightarrow
\infty}{\liminf\,}\rho(M_{n})\geq\rho(M)$.
\end{proof}

\begin{theorem}
\label{ap}Пусть $M$ $\in\mathfrak{K}\left(
A\right)  $. Если $\rho(M/\mathcal{R}_{\mathrm{hc}%
^{r}+\mathrm{cq}^{a}}\left(  A\right)  )<\rho(M)$, то
ССР
непрерывен в
$M $ и $\rho(M)=\sup\{\rho(K):K\subset M$
конечно$\}$.
\end{theorem}

\begin{proof}
По следствию
\ref{racl}, можно
считать, что $A$
--- банахова
алгебра.
Согласно
лемме \ref{sup},
найдётся
цепочка $\left(  J_{\alpha}\right)
_{\alpha\leq\gamma}$
замкнутых
идеалов,
такая что $J_{0}=0$,
$J_{\gamma}=\mathcal{R}_{\mathrm{hc}^{r}+\mathrm{cq}^{a}}\left(  A\right)  $
и все
факторы $J_{\alpha+1}/J_{\alpha}$
либо
гипокомпактны,
либо
центральны
по модулю
компактно
квазинильпотентного
радикала.
Расширяя,
если
необходимо,
цепочку за
счёт
цепочек,
соответствующих
этим
факторам, мы
можем
считать, что
все $J_{\alpha+1}/J_{\alpha}$ либо
центральны,
либо
компактно
квазинильпотентны,
либо
бикомпактны
(по
предложению
\ref{phc1}).

Найдем
первый
ординал $\beta$,
для которого
$\rho(M/J_{\beta})<\rho(M)$. Тогда $\beta$
не может
быть
предельным
по лемме \ref{cont}.
Иными
словами,
существует
ординал $\alpha$,
что $\rho(M)=\rho(M/J_{\alpha})>\rho(M/J_{\alpha+1})$.
Пусть $M^{\prime}=M/J_{\alpha}$ и
$J=J_{\alpha+1}/J_{\alpha}$. Тогда $J$
не является
компактно
квазинильпотентным
по теореме \ref{cq-rad}.
Если $J$
централен,
то ССР
непрерывен в
$M^{\prime}$ по лемме \ref{appl}
и потому --- в $M$
по лемме \ref{app}.

Остаётся
рассмотреть
случай,
когда $J$
бикомпактен.
Пусть $N=\mathrm{L}_{M^{\prime}}\mathrm{R}%
_{M^{\prime}}$. Тогда $N$ ---
предкомпактное
множество
операторов.
Так как
\[
\rho_{\chi}\left(  N\right)  \leq\rho_{e}\left(  N\right)  \leq\rho(M^{\prime
}/J)\rho(M^{\prime})<\rho(M^{\prime})^{2}=\rho(N)
\]
в силу лемм \ref{L6}
и \ref{pass}, то ССР
непрерывен в
$N$ по лемме \ref{appl}.
Пусть теперь
последовательность
ограниченных
множеств $M_{n}$
стремится к $M$.
Положим $N_{n}=\mathrm{L}%
_{M_{n}/J_{\alpha}}\mathrm{R}_{M_{n}/J_{\alpha}}$.
Тогда $N_{n}$
стремится к $N$
и
\[
\rho(M_{n})^{2}\geq\rho(M_{n}/J_{\alpha})^{2}=\rho\left(  N_{n}\right)
\rightarrow\rho(N)=\rho(M)^{2}.
\]
Следовательно,
$\lim\inf\rho(M_{n})\geq\rho(M)$ и ССР
непрерывен в
$M$. В качестве
$\left\{  M_{n}\right\}  $ можно
взять
последовательность
конечных
подмножеств,
стремящуюся
к $M$. Поэтому $\rho
(M)=\sup\rho(K)$, где $K$
пробегает
конечные
подмножества
в $M$.
\end{proof}

\begin{corollary}
Если $G$ ---
полугруппа
квазинильпотентных
элементов
нормированной
алгебры $A$ и $G\subset
\mathcal{R}_{\mathrm{hc}^{r}+\mathrm{cq}^{a}}\left(  A\right)  $,
то алгебра
$\operatorname*{span}G$
компактно квазинильпотентна.
\end{corollary}

\begin{proof}
Для каждого
конечного
подмножества
$N$ линейной
оболочки
полугруппы
найдётся
конечное
подмножество
$M\subset G$, такое что
$N\subset\mathrm{abs}\left(  M\right)  $. Так
как $\rho\left(  \mathrm{abs}\left(  M\right)  \right)
=\rho\left(  M\right)  =r\left(  M\right)  =0$, то $\rho\left(
N\right)  =0$. Поскольку
$\operatorname*{span}G$ лежит в
$\mathcal{R}_{\mathrm{hc}^{r}+\mathrm{cq}^{a}}\left(  A\right)  $,
ССР
непрерывен в
$M $ для любого
$M \in\mathfrak{K}\left(  \operatorname*{span}G\right)  $
согласно
теореме \ref{ap},
откуда ясно,
что алгебра
$\operatorname*{span}G$
компактно квазинильпотентна.
\end{proof}

Нам
неизвестно,
непрерывен
ли ССР в $M\in\mathfrak{K}\left(
A\right)  $, если
выполнено
неравенство
$\rho(M/\mathcal{R}_{\mathrm{bw}}\left(  A\right)  )<\rho(M)$.


\begin{thebibliography}{99}                                                                                               %


\bibitem {A68}J. C. Alexander, \textit{Compact Banach algebras,} Proc. London
Math. Soc. (3) \textbf{18} (1968) 1-18.

\bibitem {BW}M. A. Berger, Y. Wang, \textit{Bounded semigroups of matrices,}
Linear algebra Appl. \textbf{166} (1992), 21-27.

\bibitem {Gu}P. S. Guinand, \textit{On quasinilpotent semigroups of
operators,} Proc. Amer. Math. Soc. \textbf{86} (1982), 485-486.

\bibitem {DS}Н. Данфорд
и Дж. Шварц,
\textit{Линейные
операторы,}
том 1, ИЛ,
Москва, 1962.

\bibitem {Dix}P. G. Dixon, \textit{Topologically irreducible representations
and radicals in Banach algebras,} Proc. London Math. Soc., (3) \textbf{74}
(1997), 174-200.

\bibitem {Jun}R. Jungers, \textit{Joint spectral radius, theory and
applcations, }Springer-Verlag, Berlin, 2009.

\bibitem {LS}A. Lebow, M. Schechter, \textit{Semigroups of operators and
measures of noncompactness}, J. Funct. Anal. \textbf{7 }(1971) 1-26.

\bibitem {Lom}В. И.
Ломоносов,
\textit{Инвариантные
подпространства
для
операторов,
коммутирующих
с
компактными
операторами,}
Функц. анал. и
прилож., \textbf{7} (1973) 213-214.

\bibitem {Mor}I. D. Morris, \textit{The generalized Berger-Wang formula and
the spectral radius of linear cocycles,} preprint: ArXiv:0906.2915v1 [math.DS]
16 Jun 2009.

\bibitem {New}J. D. Newburgh, \textit{The variation of spectra}, Duke Math. J.
\textbf{18} (1951) 165-176.

\bibitem {PW}J. R. Peters, R. W. Wogen, \textit{Commutative radical operator
алгебраs,} J. Operator Theory
\textbf{42} (1999), 405-424.

\bibitem {RS}G.-C. Rota, W. G. Strang,\textit{ A note on the joint spectral
radius,} Indag. Math. \textbf{22} (1960), 379-381.

\bibitem {T85}Ю. В.
Туровский,
\textit{Спектральные
свойства
некоторых
подалгебр Ли
и
спектральный
радиус
подмножеств
в банаховых
алгебрах,}
Спектральная
теория
операторов и
её
приложения,
\textbf{6 }(1985), "Элм", Баку, 144-181.

\bibitem {ShT2001}Ю. В.
Туровский, В.
С. Шульман,
\textit{Радикалы в
банаховых
алгебрах и
некоторые
проблемы
теории
радикальных
банаховых
алгебр,}
Функц. анал. и
прилож., \textbf{35} (2001), № 4, 88-91.

\bibitem {Sh84}В. С.
Шульман,\textit{ Об
инвариантных
подпространствах
вольтерровых
операторов,}
Функц. анал. и
прилож., \textbf{18} (1984), №2, 84-85.

\bibitem {ShT2000}V. S. Shulman, Yu. V. Turovskii, \textit{Joint spectral
radius, operator semigroups and a problem of W. Wojty\'{n}ski,} J. Funct.
Anal. \textbf{177 }(2000), 383-441.

\bibitem {ShT2002}V. S. Shulman, Yu. V. Turovskii, \textit{Formulae for joint
spectral radii of sets of operators,} Studia Math. \textbf{149} (2002), 23-37.

\bibitem {rad1}V. S. Shulman, Yu. V. Turovskii, \textit{Topological radicals,
I. Basic properties, tensor products and joint quasinilpotence,} Topological
algebras, their applications и related topics, Banach Center
publications, volume \textbf{67}, pages 293-333, Warszawa 2005.

\bibitem {formula}V. S. Shulman, Yu. V. Turovskii, \textit{Application of
topological radicals to calculation of joint spectral radii,} preprint:
arXiv:0805.0209 [math.FA] 2 May 2008.

\bibitem {rad2}V. S. Shulman, Yu. V. Turovskii, \textit{Topological radicals,
II. Applications to the spectral theory of multiplication operators, }Operator
Theory: Advances and Applications, volume \textbf{212} (2010), 45-114.

\bibitem {Tur98}Yu. V. Turovskii, \textit{Volterra semigroups have invariant
subspaces,} J.Funct. Anal. \textbf{162} (2) (1999), 313-323.

\bibitem {Vala}K. Vala, \textit{On compact sets of compact operators,} Ann.
Acad. Sci. Fenn. Ser. A I \textbf{351} (1964) 1--8.
\end{thebibliography}
\end{document}